\documentclass[a4paper,11pt]{article}
\usepackage{amsthm,amsmath,amssymb}
\usepackage[OT1]{fontenc}
\usepackage{graphicx}
\usepackage[utf8]{inputenc}
\usepackage{paralist,enumitem}
\usepackage{cite}

\usepackage{tabularx,booktabs,multirow}

\usepackage[usenames,dvipsnames]{xcolor}

\newcommand\Tstrut{\rule{0pt}{3ex}}         
\newcommand\Bstrut{\rule[-1.9ex]{0pt}{0pt}}   

\newtheorem{theorem}{Theorem}
\newtheorem{lemma}{Lemma}
\newtheorem{proposition}[lemma]{Proposition}

\numberwithin{lemma}{section}

\newtheorem{observation}[theorem]{Observation}

\newtheorem{question}[theorem]{Question}
\newtheorem{definition}[theorem]{Definition}

\let\leq\leqslant
\let\geq\geqslant
\let\setminus\smallsetminus

\definecolor{our-gray}{gray}{0.40}

\newcommand{\gray}[1]{%
 \textcolor{our-gray}{#1}%
}

\newcommand{\M}[4]{%
 (%
 \arraycolsep=0.5pt%
 \begin{array}{*4c}%
  #1 & #2 & #3 & #4%
 \end{array}%
 )%
}

\newcommand{\MM}[8]{%
 \left(%
 \renewcommand\thickspace{\kern1pt}%
 \begin{smallmatrix}%
  #1 & #2 & #3 & #4\\%
  #5 & #6 & #7 & #8%
 \end{smallmatrix}%
 \right)%
}

\newcommand{\Mcross}{%
 \left(%
 \renewcommand\thickspace{\kern1pt}%
 \begin{smallmatrix}%
  - & 0 & + & 0\\%
  0 & + & - & 0\\%
  0 & - & 0 & +%
 \end{smallmatrix}%
 \right)%
}

\newcommand{\bZ}{\mathbb{Z}}

\newcommand{\bp}{\mathbf{p}}

\newcommand{\floor}[1]{{\left\lfloor #1 \right\rfloor}}
\newcommand{\ceil}[1]{{\left\lceil #1 \right\rceil}}

\newcommand{\Xa}{X_\text{\bf ind}}
\newcommand{\Xw}{X_\text{\bf cli}}

\makeatletter
\let\old@setaddresses\@setaddresses
\def\@setaddresses{\bgroup\parindent 0pt\let\scshape\relax\old@setaddresses\egroup}
\makeatother

\DeclareGraphicsExtensions{.pdf,.png,.eps}
\graphicspath{{fig/}}

\begin{document}

\title{The Chromatic Number of Ordered Graphs With Constrained Conflict Graphs}

\author{Maria Axenovich and Jonathan Rollin and Torsten Ueckerdt}

\maketitle

\begin{abstract}
 An ordered graph $G$ is a graph whose vertex set is a subset of integers.
 The edges are interpreted as tuples $(u,v)$ with $u < v$.
 For a positive integer $s$, a matrix $M \in \bZ^{s \times 4}$, and a vector $\bp = (p,\ldots,p) \in \bZ^s$ we build a conflict graph by saying that edges $(u,v)$ and $(x,y)$ are conflicting if $M(u,v,x,y)^\top \geq \bp$ or $M(x,y,u,v)^\top \geq \bp$, where the comparison is componentwise.
 This new framework generalizes many natural concepts of ordered and unordered graphs, such as the page-number, queue-number, band-width, interval chromatic number and forbidden ordered matchings.

 For fixed $M$ and $p$, we investigate how the chromatic number of $G$ depends on the structure of its conflict graph. 
 Specifically, we study the maximum chromatic number $\Xw(M,p,w)$ of ordered graphs $G$ with no $w$ pairwise conflicting edges and the maximum chromatic number $\Xa(M,p,a)$ of ordered graphs $G$ with no $a$ pairwise non-conflicting edges.
 We determine $\Xw(M,p,w)$ and $\Xa(M,p,a)$ exactly whenever $M$ consists of one row with entries in $\{-1,0,+1\}$ and moreover consider several cases in which $M$ consists of two rows or has arbitrary entries from~$\bZ$.
\end{abstract}

\section{Introduction}

At most how many colors are needed to properly color the vertices of a graph if it does not contain a fixed forbidden pattern?
This is certainly one of the most important questions in graph theory and combinatorics, where chromatic number is investigated for graphs with forbidden minors (e.g.\ Four-Color-Theorem~\cite{AH77,AHK77}), forbidden subgraphs (e.g.\ with high girth and high chromatic number~\cite{LargeGirthHighChromatic} or with given clique number and maximum degree~\cite{Reed99}), or forbidden induced subgraphs (e.g.\ perfect graphs~\cite{Ber63}), just to name a few.

In the present paper, we investigate this question for ordered graphs, that are graphs with vertices being integers, for which some information on conflicting edges is given.
The concept of conflicting edges is defined by elementary linear inequalities in terms of the edge-endpoints.
This algebraic framework captures several natural cases, such as crossing edges, nesting edges or well-separated edges, as well as some non-trivial parameters of unordered and ordered graphs, such as the queue-number, page-number, degeneracy, band-width and interval chromatic number.

Ordered graphs have been mainly investigated with respect to their ordered extremal functions~\cite{PT05,PachTardos,Klazar_SmallGraphs,Klazar_DavenportSchinzelSeq}, particularly in the case of interval chromatic number~two~\cite{FuerediHajnal,Tardos,Pettie,MarcusTardos}, and their ordered Ramsey properties~\cite{ConlonFoxLeeSudakov,BalkoCibulkaKralKyncl}.
The chromatic number of ordered graphs without a forbidden pattern has received very little attention so far; the only references being~\cite{DW04} and most recently~\cite{our-trees}.
But let us mention that, if the pattern is given by a forbidden ordered subgraph $H$ and the ordered extremal function of $H$ is linear, then there is a constant $c(H)$ such that the chromatic number of any graph without this pattern is at most $c(H)$.
In~\cite{our-trees} it is shown that even for some ordered paths $H$ there are ordered graphs of arbitrarily large chromatic number without $H$ as an ordered subgraph.

\paragraph{Ordered Graphs and Conflicting Edges.}

All considered graphs are finite, simple and have at least one edge.
An \emph{ordered graph} is a graph $G = (V,E)$ with $V \subset \bZ$, i.e., a graph whose vertices are distinct integers.
Note that here two isomorphic ordered graphs need to have exactly the same subset of $\bZ$ as their vertex set.
So this definition differs from the usual definition of ordered graphs, where only the ordering of the vertices matters but not an embedding into $\bZ$.
We consider the integers, and therefore the vertices of $G$, laid out along a horizontal line ordered by increasing value from left to right.
Hence if $u,v \in V \subset \bZ$, $u < v$, we say that $u$ is \emph{left of} $v$ and $v$ is \emph{right of} $u$.
For a fixed ordered graph $G = (V,E)$, an edge $e \in E$ is then associated with the (ordered) tuple $(u,v)$ where $e = uv$ and $u < v$.
For a positive integer $s$ and two vectors $x,y \in \bZ^s$, $x + y$ denotes the componentwise addition, and $x \leq y$ and $x \geq y$ denote the standard componentwise comparability of $x$ and $y$.
We shall abbreviate the vector $(p,\ldots,p)^\top \in \bZ^s$ by $\bp$.
For a given injective map $\phi:\bZ\to\bZ$ and an ordered graph $G$ we say that an ordered graph $G'$ is \emph{obtained from $G$ by $\phi$} if $V(G') = \{\phi(x) \mid x \in V(G)\}$ and $(\phi(x),\phi(y))$ is an edge in $G'$ if and only if $(x,y)$ is an edge in $G$ for any $x$, $y\in\bZ$.
For example, from an ordered graph we obtain another ordered graph by translating or scaling the vertex set.

For a matrix $M \in \bZ^{s \times 4}$ and a parameter $p \in \bZ$ we define the \emph{conflict graph} of $G$ with respect to $M$ and $p$, denoted by $M_p(G)$, as follows:
\begin{align*}
 V(M_p(G)) &:= E(G),\\
 E(M_p(G)) &:= \{ e_1e_2 \mid e_1 = (u_1,v_1); e_2 = (u_2,v_2); e_1,e_2 \in E;\\
  & \hspace{4.5em} M(u_1,v_1,u_2,v_2)^\top \geq \bp \text{ or } M(u_2,v_2,u_1,v_1)^\top \geq \bp\}.
\end{align*}
We say that $e_1,e_2 \in E(G)$ are \emph{conflicting} if $e_1e_2 \in E(M_p(G))$.
Let $M'$ denote the matrix that is obtained from $M$ by swapping the first column with the third and the second with the fourth.
Note that this operation preserves all conflicts and non-conflicts, and hence $M_p(G)=M'_p(G)$.

In many cases considered here the matrix $M$ has entries in $\{-1,0,1\}$.
For better readability we shall use the symbols $\{-,0,+\}$ instead of $\{-1,0,1\}$ as the entries of $M$.

One advantage of this proposed abstract framework is that many natural parameters of an (unordered) graph $F$ can conveniently be phrased in terms of $M_p(G)$, where $G$ is an ordered graph whose underlying unordered graph is $F$.
For example, two edges $e_1=(u_1,v_1)$, $e_2=(u_2,v_2)$ in an ordered graph $G$ are called \emph{crossing} if $u_1<u_2<v_1<v_2$.
Similarly, an edge $(u_1,v_1)$ is \emph{nested under} an edge $(u_2,v_2)$ if $u_2 < u_1 < v_1 < v_2$.
Then $e_1$, $e_2$ are crossing, respectively nesting, if and only if $e_1$, $e_2$ are conflicting with respect to $p=1$ and
\[
 M^\text{cross} = \Mcross, \quad \text{respectively} \quad  M^\text{nest}=\MM{+}{0}{-}{0}{0}{-}{0}{+},
\]
see Figure~\ref{fig:Example-Matrices} (top left and center).
So, there is no pair of crossing edges in an ordered graph $G$ if and only if $\omega(M^\text{cross}_1(G)) = 1$ and there is no set of $w+1$ pairwise crossing edges if and only if $\omega(M^\text{cross}_1(G)) \leq w$, where $\omega(H)$ denotes the clique number of graph $H$.
The former case characterizes outerplanar graphs and the latter case was considered by Capoyleas and Pach~\cite{CapoyleasPach}, who showed that every $n$-vertex ordered graph $G$ with $\omega(M^\text{cross}_1(G)) \leq w$ has at most $2wn-\binom{2w+1}{2}$ edges.
In Section~\ref{sec:other-parameters} we give further examples of graph parameters that can be phrased in terms of $M_p(G)$ for appropriate $M$ and $p$.

\begin{figure}[htb]
 \centering
 \includegraphics{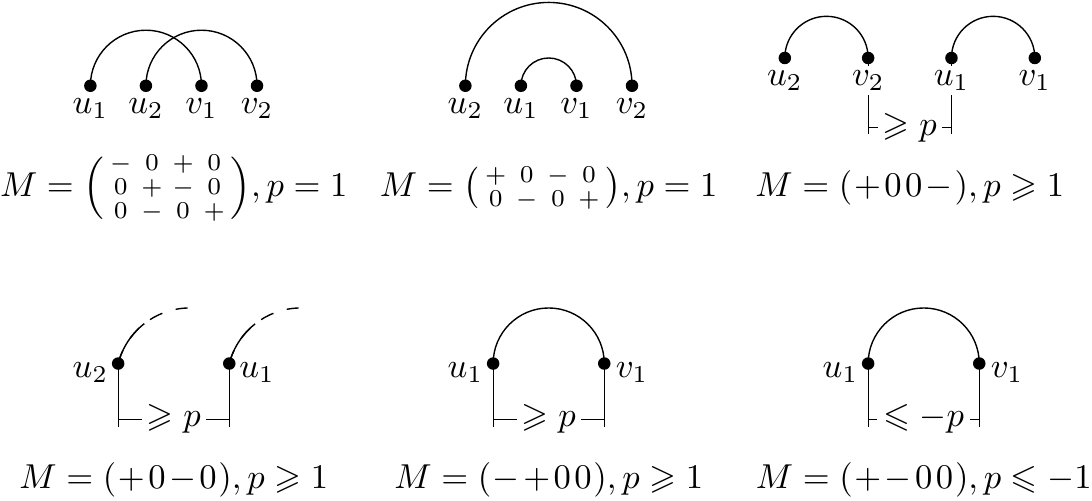}
 \caption{Examples of conflicting edges $(u_1,v_1)$, $(u_2,v_2)$ with respect to different matrices $M$ and parameter $p$. The top-left shows crossing edges and the top-center shows nesting edges. In the bottom-center and bottom-right the second edge $(u_2,v_2)$ is irrelevant.}
 \label{fig:Example-Matrices}
\end{figure}


In this paper, we are interested in the relation between the chromatic number $\chi(G)$ of $G$ and basic graph parameters of $M_p(G)$, such as its independence number $\alpha(M_p(G))$ and its clique number $\omega(M_p(G))$.
Specifically, we investigate whether a high chromatic number implies the existence of a large set of pairwise conflicting or pairwise non-conflicting edges.

\begin{definition}\label{def:sup-chi}
  Let $p$, $s$, $a$, $w \in \bZ$, with $s$, $a$, $w \geq 1$, and $M \in \bZ^{s \times 4}$. Then
  \begin{align*}
  & \Xa(M,p,a) = \sup\{\chi(G) \mid G \text{ ordered graph with } \alpha(M_p(G)) \leq a\}\\
  \text{ and } & \Xw(M,p,w) = \sup\{\chi(G) \mid G \text{ ordered graph with } \omega(M_p(G)) \leq w\}.
 \end{align*}
\end{definition}

For example, from the discussion above we have $\Xw(M^\text{cross},1,1) = 3$, as outerplanar graphs are $3$-colorable, and $\Xw(M^\text{cross},1,w) \leq 4w$ for $w\geq 2$, as ordered graphs with no $w+1$ pairwise crossing edges are $(4w-1)$-degenerate.

For $x \in \bZ$, $x \geq 1$, let\footnote{For $x\geq 1$ we have $\binom{\sqrt{2x}}{2} < x < \binom{\sqrt{2x}+1}{2}$.
Thus $f(x)=\floor{\sqrt{2x}}$ if $x < \binom{\scriptstyle\floor{\sqrt{2x}}+1}{2}$ and $f(x)=\floor{\sqrt{2x}}+1$ otherwise.} $f(x)$ be the largest integer $k$ with $\binom{k}{2}\leq x$.
For any $k \geq 2$ the complete graph $K_k$ is a $k$-chromatic graph with only $\binom{k}{2}$ edges.
Therefore for any $M \in \bZ^{s \times 4}$, $p\in\bZ$, and $k\geq 2$, we have $\alpha(M_p(K_k))$, $\omega(M_p(K_k))\leq |V(M_p(K_k))| = |E(K_k)| = \binom{k}{2}$ and thus
\begin{equation}
 \Xa(M,p,a)\geq f(a) \text{ and } \Xw(M,p,w)\geq f(w).\label{eq:X-lower-bound}
\end{equation}
We shall prove that the lower bounds in~\eqref{eq:X-lower-bound} are attained for some matrices $M$ and parameters $p$.
On the other hand, there is no general upper bound, as we shall show that $\Xa(M,p,a)=\infty$ or $\Xw(M,p,w)=\infty$ for some other matrices $M$ and parameters $p$.

As it turns out, instead of studying the functions $\Xa(M,p,a)$ and $\Xw(M,p,w)$ directly, it is often more convenient to consider their integral inverses, i.e., we consider the functions $A(M,p,k)$ and $W(M,p,k)$ defined as follows.

\begin{definition}\label{def:min-alpha}
 Let $p$, $s$, $k \in\bZ$, with $s\geq 1$, $k\geq 2$, and $M \in \bZ^{s \times 4}$. Then
 \begin{align*}
  & A(M,p,k) := \min\{\alpha(M_p(G)) \mid G \text{ ordered graph with } \chi(G) \geq k\}\\
  \text{ and } & W(M,p,k) := \min\{\omega(M_p(G)) \mid G \text{ ordered graph with } \chi(G) \geq k\}.
 \end{align*}
\end{definition}
Note that replacing the minima by maxima in Definition~\ref{def:min-alpha} is not interesting since it is almost always easy to construct bipartite ordered graphs with many pairwise conflicting and pairwise non-conflicting edges.
By considering $M_p(K_k)$, similar to above, one obtains the following bounds for any $M \in \bZ^{s \times 4}$, any $p \in \bZ$, and any $k\geq 2$
\begin{equation}
 1 \leq A(M,p,k),\ W(M,p,k) \leq \binom{k}{2}.\label{eq:general-bounds}
\end{equation}
Since the conflict graph of an ordered graph without any edges has no vertices, we exclude the case $k=1$ throughout.
While the functions from Definition~\ref{def:sup-chi} yield the smallest $k$ such that all ordered graphs without a certain pattern $P$ can be colored with less than $k$ colors, the functions from Definition~\ref{def:min-alpha} address the contraposition, namely whether every graph with chromatic number at least $k$ necessarily contains the pattern $P$.
For example, instead of proving $\Xw(M^\text{cross},1,1) \leq 3$, i.e., that every outerplanar graph is $3$-colorable, one can equivalently prove that $W(M^\text{cross},1,4) \geq 2$, i.e., that every non-$3$-colorable ordered graph has a pair of crossing edges.

More generally, $\Xa(M,p,a)$ and $A(M,p,k)$ are related by
\begin{equation}
 \Xa(M,p,a) = \sup\{k \mid A(M,p,k) \leq a\} \text{ for fixed } M,p,a,\label{eq:X-from-A}
\end{equation}
while $\Xw(M,p,w)$ and $W(M,p,k)$ are related by
\begin{equation}
 \Xw(M,p,w) = \sup\{k \mid W(M,p,k) \leq w\} \text{ for fixed } M,p,w.\label{eq:X-from-W}
\end{equation}
The advantage of the functions $A$ and $W$ is that they can be nicely expressed as polynomial type functions of $k$ and rational type functions of $p$ (see Table~\ref{tab:detailed}).

\paragraph{Our Results.}

For many matrices $M$ and parameters $p$ it turns out that $A(M,p,k)$ or $W(M,p,k)$ or both attain the lower or upper bound in~\eqref{eq:general-bounds} for all $k \geq 2$.
Focusing on $1 \times 4$--matrices, the calculation of $A(M,p,k)$ and $W(M,p,k)$ becomes non-trivial only for quite specific matrices (see first statement in Theorem~\ref{thm:generalMatrices}).
We say that a matrix $M \in \bZ^{s \times 4}$ is \emph{translation invariant} if for any vector $x \in \bZ^4$, any $p \in \bZ$ and any $t \in \bZ$ we have
\[
 Mx \leq \bp \quad \Leftrightarrow \quad M(x + \mathbf{t}) \leq \bp.
\]
Intuitively speaking, $M$ is translation invariant, if whether or not two edges are conflicting does not depend on the absolute coordinates of their endpoints, rather than their relative position to one another.
For example, when $M$ is translation invariant, then for any ordered graph $G$ and any $t \in \bZ$ we have $M_p(G) = M_p(G_t)$ where $G_t$ arises from $G$ by shifting all vertices $t$ positions to the right if $t \geq 0$, respectively $|t|$ positions to the left if $t < 0$.

As $M(x + \mathbf{t}) = Mx + M\mathbf{t}$ we immediately get the following algebraic characterization of translation invariance.

\begin{observation}
 A matrix $M \in \bZ^{s \times 4}$ is translation invariant if and only if $M\mathbf{1} = \mathbf{0}$.
\end{observation}

In other words, a matrix $M$ is translation invariant if and only if in each row of $M$ the entries sum to $0$.
We give several conditions for matrices $M$ and parameters $p$ under which $A(M,p,k)$ or $W(M,p,k)$ or both attain the lower or upper bound in~\eqref{eq:general-bounds}.

\begin{theorem}\label{thm:generalMatrices}
 Let $s$, $k$, $p\in\bZ$, with $s\geq 1$, $k\geq 2$, and $M \in \bZ^{s \times 4}$.
 If $M$ is not translation invariant, then $W(M,p,k) = 1$.
 Moreover, if $M\mathbf{1}>\mathbf{0}$ or $M\mathbf{1}<\mathbf{0}$, then $A(M,p,k) = 1$.
 
 \noindent If $M=\M{m_1~}{m_2~}{m_3~}{m_4}\in\bZ^{1\times 4}$ is translation invariant, then each of the following holds.
 
 \begin{enumerate}[label = (\roman*)]
  \item If $m_2 + m_4 \geq\max\{p,0\}$, then $A(M,p,k) = 1$ and $W(M,p,k)=\binom{k}{2}$.\label{enum:m2m40}
  
  \item If $m_1+m_2 = m_3+m_4 = 0$, $m_2, m_4 \leq 0$ and $p > m_2+m_4$, then $A(M,p,k) = \binom{k}{2}$ and $W(M,p,k)=1$.\label{enum:m1m20}
  
  \item If $m_2+m_4 > 0$ or ($m_2+m_4 = 0$ and $m_1,m_2 \neq 0$), then $A(M,p,k)=1$.\label{enum:m2m4greater0}
    
  \item If $m_2,m_4 < 0$, then $W(M,p,k)=1$.\label{enum:m2m4smaller0}
 \end{enumerate}
 
 \noindent
 In all cases above we have $A(M,p,k) = \alpha(M_p(K_k))$ for some ordered graph $K_k$ and $W(M,p,k) = \omega(M_p(K_k))$ for some ordered graph $K_k$.
\end{theorem}

In case $M$ and $p$ do not satisfy any of the requirements of Theorem~\ref{thm:generalMatrices}, the exact behavior of $A(M,p,k)$ and $W(M,p,k)$ can be non-trivial.
We determine $A(M,p,k)$ and $W(M,p,k)$ exactly for all $M \in \{-1,0,1\}^{1 \times 4}$ and $p,k \in \bZ$, $k \geq 2$.

\begin{theorem}\label{thm:specialMatrices}
 For all $p$, $k\in\bZ$, $k \geq |p|+2$ and matrices $M\in\{-1,0,1\}^{1\times 4}$ we have $A(M,p,k) = \alpha(M_p(K_k))$ for some ordered graph $K_k$ and $W(M,p,k) = \omega(M_p(K_k))$ for some ordered graph $K_k$.
 
 The exact values of $A(M,p,k)$ and $W(M,p,k)$ are given in Table~\ref{tab:detailed}.
\end{theorem}

Whenever $k < |p|+2$ the exact values follow from Theorem~\ref{thm:generalMatrices} or are given in Propositions~\ref{P0M0}~--~\ref{M00P}.
Finally, we consider the $2 \times 4$--matrix $M^\text{nest} = \MM{+}{0}{-}{0}{0}{-}{0}{+}$ that is related to nesting edges, see Figure~\ref{fig:Example-Matrices} top middle.
Dujmovi\'{c} and Wood~\cite{DW04} give upper and lower bounds on $\Xw(M^\text{nest},1,w)$ and ask for the exact value.

\begin{theorem}\label{thm:nest}
 Let $M = \MM{+}{0}{-}{0}{0}{-}{0}{+}$.
 \begin{compactitem}[\enskip]
  \item If $p \geq 1$, then  $A(M,p,k) = 2k-3$ and $\frac{k}{4p} \leq W(M,p,k) \leq  \ceil{\frac{k-1}{2p}}$ for all $k \geq 2$.
 
  \item If $p \leq 0$, then $A(M,p,k)=\ceil{\frac{k-1}{1-p}}$ and $W(M,p,k) = k-1$ for all $k \geq 2$.
 \end{compactitem}
\end{theorem}

\begin{table}[htbp]
 \small
 \centering
 \begin{tabular}{|c|c|c|c|}
  \hline
  $M$ & $p$ & $A(M,p,k)$ & $W(M,p,k)$ \Tstrut\Bstrut\\
  
  \hline
  \hline
  
  $M\mathbf{1}\neq 0$ & $p\in\bZ$  & \gray{$1$} & \gray{$1$} \Tstrut\Bstrut\\
  
  \hline
  \hline
  
  \multirow{2}{*}{$\M{0}{0}{0}{0}$}
  & $p \leq 0$ & \gray{$1$} & \gray{$\binom{k}{2}$} \Tstrut\\
  & $p > 0$    & \gray{$\binom{k}{2}$} & \gray{$1$} \Tstrut\Bstrut\\
  
  \hline
  
  $\M{+}{0}{-}{0}^\star$ & \multirow{2}{*}{$p \leq 0$} & \multirow{2}{*}{\gray{$1$}} & \multirow{2}{*}{\gray{$\binom{k}{2}$}} \Tstrut\\
  $\M{-}{0}{+}{0}$ & & & \\
  $\M{0}{+}{0}{-}$ & \multirow{2}{*}{$p > 0$} & \multirow{2}{*}{$k - 1$} & \multirow{2}{*}{$\ceil{\frac{k-1}{p}}$} \\
  $\M{0}{-}{0}{+}$ & & & \Bstrut\\
 
  \hline
 
  $\M{-}{+}{0}{0}^\star$ & $p \leq 1$ & \gray{$1$} & \gray{$\binom{k}{2}$} \Tstrut\\
  $\M{0}{0}{-}{+}$ & $p \geq 1$ & \gray{$1$} & $1+ \binom{k-p+1}{2}$ \Tstrut\Bstrut\\
  
  \hline
  
  $\M{+}{-}{0}{0}^\star$ & $p \leq 0$ & $\binom{k+p}{2}$ & \gray{$1$} \Tstrut\\
  $\M{0}{0}{+}{-}$ & $p \geq 0$ & \gray{$\binom{k}{2}$} & \gray{$1$} \Tstrut\Bstrut\\
  
  \hline
  
  \multirow{2}{*}{$\M{-}{+}{-}{+}^\star$}
  & $p \leq 2$ & \gray{$1$} & \gray{$\binom{k}{2}$} \Tstrut\\
  & $p \geq 2$ & \gray{$1$} & $p \bmod 2 + \binom{k-\ceil{\frac{p}{2}}+1}{2}$ \Tstrut\Bstrut\\
  
  \hline
  
  \multirow{2}{*}{$\M{+}{-}{+}{-}$}
  & $p \leq -1$ & $(1-p) \bmod 2 + \binom{k-\ceil{\frac{1-p}{2}}+1}{2}$ & \gray{$1$} \Tstrut\\
  & $p \geq -1$ & \gray{$\binom{k}{2}$} & \gray{$1$} \Tstrut\Bstrut\\
  
  \hline
  
  $\M{+}{-}{-}{+}^\star$ & $p \leq 0$ & \gray{$1$} & \gray{$\binom{k}{2}$} \Tstrut\\
  $\M{-}{+}{+}{-}$ & $p > 0$ & \gray{$1$} & $\ceil{\frac{k-1}{p}}$ \Tstrut\Bstrut\\
  
  \hline
  
  $\M{+}{+}{-}{-}^\star$ & $p \leq 0$ & \gray{$1$} & \gray{$\binom{k}{2}$} \Tstrut\\
  $\M{-}{-}{+}{+}$ & $p > 0$ & \gray{$1$} & $\ceil{\frac{2k-3}{p}}$ \Tstrut\Bstrut\\
  
  \hline
  
  $\M{+}{0}{0}{-}^\star$ & $p \leq 0$ & $\floor{\frac{(k+p)^2}{4}}$ & $k-1$ \Tstrut\\
  $\M{0}{-}{+}{0}$ & $p > 0$ & $\floor{\frac{(k+1)^2}{4}}-1$ & $\ceil{\frac{k-1}{p+1}}$ \Tstrut\Bstrut\\
  
  \hline
  
  $\M{-}{0}{0}{+}^\star$ & $p \leq 1$ & \gray{$1$} & \gray{$\binom{k}{2}$} \Tstrut\\
  $\M{0}{+}{-}{0}$ & $p \geq 2$ & \gray{$1$} & $\binom{k-p+2}{2}$ \Tstrut\Bstrut\\
  
  \hline
  \hline
  
  \multirow{2}{*}{$\MM{+}{0}{-}{0}{0}{+}{0}{-}^\star$}
  & $p \leq 0$ & $\frac{k}{4(1-p)} \leq \cdot \leq \ceil{\frac{k-1}{2(1-p)}}$ & $2k-3$ \Tstrut\\
  & $p > 0$ & $k-1$ & $\ceil{\frac{k-1}{p}}$ \Tstrut\Bstrut\\
  
  \hline
  
  \multirow{2}{*}{$\MM{+}{0}{-}{0}{0}{-}{0}{+}$}
  & $p \leq 0$ & $\ceil{\frac{k-1}{1-p}}$ & $k-1$ \Tstrut\\
  & $p > 0$ & $2k-3$ & $\frac{k}{4p} \leq \cdot \leq \ceil{\frac{k-1}{2p}}$ \Tstrut\Bstrut\\
  \hline
 \end{tabular}
 \caption{\small Values of $A(M,p,k)$ and $W(M,p,k)$ for $p$, $k\in\bZ$, $k \geq |p|+2$ and matrices $M$.
 The first row covers all non-translation invariant $M \in \bZ^{1 \times 4}$, rows 2nd to 11th cover all translation-invariant $M\in\{-1,0,1\}^{1\times 4}$, and the last two rows cover two $M \in \{-1,0,1\}^{2 \times 4}$. Gray entries follow from Theorem~\ref{thm:generalMatrices}, results in last two rows follow from Theorem~\ref{thm:nest} and Proposition~\ref{P0M00P0M} in Section~\ref{sec:nest}; remaining entries are proven in Propositions~\ref{P0M0}~--~\ref{M00P} in Section~\ref{sec:proofs} using the matrices marked with $^\star$.}
 \label{tab:detailed}
\end{table}

The values and bounds for $\Xa(M,p,a)$ and $\Xw(M,p,w)$ corresponding to the results above are calculated using the identities~\eqref{eq:X-from-A} and~\eqref{eq:X-from-W} and given in Table~\ref{tab:Xdetailed}.
By definition of $f(x)$, the upper or lower bounds in~\eqref{eq:general-bounds} translate as follows.

\smallskip

\begin{compactitem}[\leftmargin = 0pt]
 \item If $A(M,p,k) = 1$ for all $k \geq 2$, then $\Xa(M,p,a) = \infty$ for all $a \geq 1$.
 \item If $A(M,p,k) = \binom{k}{2}$ for all $k \geq 2$, then $\Xa(M,p,a) = f(a)$ for all $a \geq 1$.
 \item If $W(M,p,k) = 1$ for all $k \geq 2$, then $\Xw(M,p,w) = \infty$ for all $w \geq 1$.
 \item If $W(M,p,k) = \binom{k}{2}$ for all $k \geq 2$, then $\Xw(M,p,w) = f(w)$ for all $w \geq 1$.
\end{compactitem}

\smallskip

\begin{table}[htbp]
 \small
 \centering
 \begin{tabular}{|c|c|c|c|}
  \hline
  $M$ & $p$ & $\Xa(M,p,a)$ & $\Xw(M,p,w)$ \Tstrut\Bstrut\\
  
  \hline
  \hline
  
  $M\mathbf{1}\neq 0$ & $p\in\bZ$  & \gray{$\infty$} & \gray{$\infty$} \Tstrut\Bstrut\\
    
  \hline
  \hline
  
  \multirow{2}{*}{$\M{0}{0}{0}{0}$}
  & $p \leq 0$ & \gray{$\infty$} & \gray{$f(w)$} \Tstrut\\
  & $p > 0$    & \gray{$f(a)$} & \gray{$\infty$} \Tstrut\Bstrut\\
  
  \hline
  
  $\M{+}{0}{-}{0}$ & \multirow{2}{*}{$p \leq 0$} & \multirow{2}{*}{\gray{$\infty$}} & \multirow{2}{*}{\gray{$f(w)$}} \Tstrut\\
  $\M{-}{0}{+}{0}$ & & & \\
  $\M{0}{+}{0}{-}$ & \multirow{2}{*}{$p > 0$} & \multirow{2}{*}{$a + 1$} & \multirow{2}{*}{$pw+1.$} \\
  $\M{0}{-}{0}{+}$ & & & \Bstrut\\
 
  \hline
 
  $\M{-}{+}{0}{0}$ & $p \leq 1$ & \gray{$\infty$} & \gray{$f(w)$} \Tstrut\\
  $\M{0}{0}{-}{+}$ & $p \geq 1$ & \gray{$\infty$} & $f(w-1)+p-1$ \Tstrut\Bstrut\\
  
  \hline
  
  $\M{+}{-}{0}{0}$ & $p \leq 0$ & $f(a)-p$ & \gray{$\infty$} \Tstrut\\
  $\M{0}{0}{+}{-}$ & $p \geq 0$ & \gray{$f(a)$} & \gray{$\infty$} \Tstrut\Bstrut\\
  
  \hline
  
  \multirow{2}{*}{$\M{-}{+}{-}{+}$}
  & $p \leq 2$ & \gray{$\infty$} & \gray{$f(w)$} \Tstrut\\
  & $p \geq 2$ & \gray{$\infty$} & {\footnotesize $f(w-p\bmod{2})+\ceil{\frac{p-2}{2}}$} \Tstrut\Bstrut\\
  
  \hline
  
  \multirow{2}{*}{$\M{+}{-}{+}{-}$}
  & $p \leq -1$ & {\footnotesize $f(a-(1-p)\bmod{2})+\ceil{\frac{-(p+1)}{2}}$} & \gray{$\infty$} \Tstrut\\
  & $p \geq -1$ & \gray{$f(a)$} & \gray{$\infty$} \Tstrut\Bstrut\\
  
  \hline
  
  $\M{+}{-}{-}{+}$ & $p \leq 0$ & \gray{$\infty$} & \gray{$f(w)$} \Tstrut\\
  $\M{-}{+}{+}{-}$ & $p > 0$ & \gray{$\infty$} & $pw+1$ \Tstrut\Bstrut\\
  
  \hline
  
  $\M{+}{+}{-}{-}$ & $p \leq 0$ & \gray{$\infty$} & \gray{$f(w)$} \Tstrut\\
  $\M{-}{-}{+}{+}$ & $p > 0$ & \gray{$\infty$} & $\floor{\frac{pw+3}{2}}$ \Tstrut\Bstrut\\
  
  \hline
  
  $\M{+}{0}{0}{-}$ & $p \leq 0$ & $\floor{\sqrt{4a+3}}-p$ & $w+1$ \Tstrut\\
  $\M{0}{-}{+}{0}$ & $p \geq 0$ & $\floor{\sqrt{4a+7}}-1$ & $(p+1)w+1$ \Tstrut\Bstrut\\
  
  \hline
  
  $\M{-}{0}{0}{+}$ & $p \leq 1$ & \gray{$\infty$} & \gray{$f(w)$} \Tstrut\\
  $\M{0}{+}{-}{0}$ & $p \geq 2$ & \gray{$\infty$} & $f(w)+p-2$ \Tstrut\Bstrut\\
  
  \hline
  \hline
  
  \multirow{2}{*}{$\MM{+}{0}{-}{0}{0}{+}{0}{-}$}
  & $p \leq 0$ & $2(1-p)a+1 \leq \cdot \leq 4(1-p)a$ & $\floor{\frac{w+3}{2}}$ \Tstrut\\
  & $p > 0$ & $a+1$ & $pw+1$ \Tstrut\Bstrut\\

  \hline
  
  \multirow{2}{*}{$\MM{+}{0}{-}{0}{0}{-}{0}{+}$}
  & $p \leq 0$ & $(1-p)a+1$ & $w+1$ \Tstrut\\
  & $p > 0$ & $\floor{\frac{a+3}{2}}$ & $ 2pw+1 \leq \cdot \leq 4pw$ \Tstrut\Bstrut\\
  \hline
 \end{tabular} 
 \caption{\small Values of $\Xa(M,p,a)$ and $\Xw(M,p,w)$ for $p$, $a$, $w\in\bZ$, $a$, $w\geq 1$, and matrices $M$.
 The first row covers all non-translation invariant $M \in \bZ^{1 \times 4}$, rows 2nd to 11th cover all translation-invariant $M\in\{-1,0,1\}^{1\times 4}$, and the last two rows cover two $M \in \{-1,0,1\}^{2 \times 4}$.
 Recall that $f(x)$ is the largest integer $k$ such that $\binom{k}{2}\leq x$. The results follow from the results in Table~\ref{tab:detailed} using the identities~\eqref{eq:X-from-A} and~\eqref{eq:X-from-W}.}
 \label{tab:Xdetailed}
\end{table}

\paragraph{Organization of the Paper.}

In Section~\ref{sec:other-parameters} we show how several graph parameters can be phrased in terms of $M_p(G)$ for appropriate $M$ and $p$.
In Section~\ref{sec:proof-general-thm} we prove Theorem~\ref{thm:generalMatrices}.
In Section~\ref{sec:proofs} we prove Theorem~\ref{thm:specialMatrices}.
Here we determine $A(M,p,k)$ and $W(M,p,k)$ exactly for all $p$, $k \in \bZ$, $k \geq 2$, and all translation invariant matrices $M \in \{-1,0,1\}^{1 \times 4}$ in Propositions~\ref{P0M0}~--~\ref{M00P}.
Prior to that, we provide some lemmas in Section~\ref{sec:prelim} which enable us to restrict our attention to only eight such translation invariant $1 \times 4$--matrices. 
In Section~\ref{sec:nest} we prove Theorem~\ref{thm:nest}.
Finally we give conclusions and further questions in Section~\ref{sec:conclusions}.

\paragraph{Notation.}
For a positive integer $n$ we write $[n]=\{1,\ldots,n\}$.

\section{Relation to other graph parameters}\label{sec:other-parameters}

Further examples of graph parameters that can be phrased in terms of $M_p(G)$ include the page-number $p(F)$~\cite{Oll73}, queue-number $q(F)$~\cite{HR92}, degeneracy $d(F)$~\cite{LW70}, and band-width $b(F)$~\cite{Har64,Kor66} of a graph $F$.
The \emph{page-number} (respectively \emph{queue-number}) of $F$ is the minimum $k$ for which there exists a vertex-ordering and a partition of the edges into $k$ sets $S_1,\ldots,S_k$ such that no two edges in the same $S_i$, $i=1,\ldots,k$, are crossing (respectively nesting).
Denoting by $G^\star$ the underlying unordered graph of a given ordered graph $G$, we have\footnote{For $M = \MM{+}{0}{-}{0}{0}{-}{0}{+}$ and $p\geq 1$ we have for any $G$ that $M_p(G)$ is a comparability graph with respect to the relation ``nested under'', and thus $\omega(M_p(G)) = \chi(M_p(G))$.}
\begin{align*}
 p(F) & = \min \{\chi(M_p(G)) \mid G^\star = F\}, \text{ for } M = \Mcross,\ p = 1,\\
 q(F) & = \min \{\chi(M_p(G)) \mid G^\star = F\}, \text{ for } M = \MM{+}{0}{-}{0}{0}{-}{0}{+},\ p = 1. 
\end{align*}
The \emph{degeneracy} (respectively \emph{band-width}) of $F$ is the smallest $k$ for which there exists a vertex-ordering such that every vertex has at most $k$ neighbors with a smaller index (respectively every edge has length at most $k$).
Here the \emph{length of an edge $(u,v)$} in an ordered graph is given by $v-u$.
In our framework we can write degeneracy $d$ and band-width $b$ as
\begin{align*}
 d(F) & = \min \{\omega(M_p(G)) \mid G^\star = F\}, \text{ for } M = \MM{0}{+}{0}{-}{0}{-}{0}{+},\ p = 0,\\
 b(F) & = \min \{p \geq 1 \mid \omega(M_{p+1}(G)) = 1, G^\star = F\}, \text{ for } M = \M{-}{+}{0}{0}.
\end{align*}
Moreover, if $G$ is an ordered graph, then its \emph{interval chromatic number $\chi_\prec(G)$} is the minimum number of intervals $\bZ$ can be partitioned into, so that no two vertices in the same interval are adjacent in $G$~\cite{PT05}.
As we shall prove later (c.f.\ Lemma~\ref{lem:a-almost-t-colorable}), this can be rephrased as $\chi_\prec(G) = \omega(M_p(G))+1$, for $M = \M{+}{0}{0}{-}$ and $p = 0$.

Let us also mention that Dujmovi\'{c} and Wood~\cite{DW04} define for $k \in \bZ$, $k \geq 2$, a \emph{$k$-edge necklace} in an ordered graph $G$ as a set of $k$ edges of $G$ which are pairwise in conflict with respect to $M = \M{+}{0}{0}{-}$ and $p = 1$, see Figure~\ref{fig:Example-Matrices} top right.
They further define the \emph{arch-number} of a graph $F$ as
\[
 \operatorname{an}(F) = \min \{ \omega(M_p(G)) \mid G^\star = F\} , \text{ for } M = \M{+}{0}{0}{-},\ p = 1,
\]
and prove that the largest chromatic number among all graphs $F$ with $\operatorname{an}(F)\leq w$ equals $2w+1$.
We generalize this result to any $p \in \bZ$ (c.f.\ Proposition~\ref{P00M}).

\section{Proof of Theorem~\ref{thm:generalMatrices}}\label{sec:proof-general-thm}

 Let $s$, $k$, $p\in\bZ$, with $s\geq 1$, $k\geq 2$, and $M \in \bZ^{s \times 4}$.

\medskip

 First of all assume that $M$ is not translation invariant.
 Let $G$ be any ordered graph with $\chi(G)\geq k$.
 For an integer $t$, let $G_t$ denote the ordered graph obtained from $G$ by adding $t$ to every vertex (so $G_t$ contains an edge $(x+t,y+t)$ if and only if $(x,y)$ is an edge in $G$).
 Clearly, we have $M(x + \mathbf{t}) = Mx + M\mathbf{t} = Mx + t(M\mathbf{1})$ for any vector $x \in \bZ^4$.
 Hence,  since $M\mathbf{1} \neq \mathbf{0}$, there is a large or small enough $t = t(G,M,p)$ such that for the first row $M^{(1)}$ of $M$ and any $u_1,v_1,u_2,v_2 \in V(G)$ we have $M^{(1)}(u_1,v_1,u_2,v_2)^\top < p$, i.e., the conflict graph $M_p(G_t)$ is empty and its clique number is~$1$.
 This implies that $W(M,p,k)=1$.

 If additionally $M\mathbf{1}>\mathbf{0}$ (respectively $M\mathbf{1}<\mathbf{0}$), then there is a large (respectively small) enough $t = t(G,M,p)$ such that $M_p(G_t)$ is complete, implying that $A(M,p,k)=1$.
 
\medskip

 Now assume that $M=\M{m_1~}{m_2~}{m_3~}{m_4}\in\bZ^{1\times 4}$ is translation invariant.
 Consider an ordered graph $G$ with $\chi(G)\geq k$.
 
 \begin{enumerate}[label = (\roman*)] 
  \item 
  We assume that $m_2 + m_4 \geq\max\{p,0\}$.
  Consider edges $(u_1,v_1)$ and $(u_2,v_2)$ in $G$.
 Then $v_1+v_2 \geq u_1+u_2+2$ and thus
 \begin{align*}
  &\hspace{-5em} M(u_1,v_1,u_2,v_2)^\top + M(u_2,v_2,u_1,v_1)^\top\\
  =\, &(m_1+m_3)(u_1+u_2)+(m_2+m_4)(v_1+v_2)\\
  \geq\, &(m_1+m_3)(u_1+u_2)+(m_2+m_4)(u_1+u_2+2)\\
  =\, &(m_1+m_2+m_3+m_4)(u_1+u_2)+2(m_2+m_4)\\
  =\, &2(m_2+m_4)\\
  \geq\, &2p.
 \end{align*}
 Hence we have $M(u_1,v_1,u_2,v_2)^\top \geq p$ or $M(u_2,v_2,u_1,v_1)^\top \geq p$.
 Therefore $(u_1,v_1)$ and $(u_2,v_2)$ are conflicting and $M_p(G)$ is a complete graph on $|E(G)|$ vertices.
 This clearly implies that $A(M,p,k)=1$.
 
 Secondly, since $\chi(G) \geq k$, we have $W(M,p,k)\geq |E(G)| \geq \binom{k}{2}$, since there is an edge between any two color classes in an optimal proper coloring of $G$.
 Moreover $W(M,p,k)\leq \omega(M_p(K_k)) = \binom{k}{2}$, see~\eqref{eq:general-bounds}.
 This shows that $W(M,p,k) = \binom{k}{2}$.

 \item
 We assume that $m_1+m_2 = m_3+m_4 = 0$, $m_2,m_4 \leq 0$ and $p > m_2+m_4$.
 For any edge $(u,v)$ in $G$ we have $u + 1 \leq v$ and thus $m_2(u+1) \geq m_2v$ and $m_4(u+1) \geq m_4v$.
 Hence for any two edges $(u_1,v_1)$ and $(u_2,v_2)$ in $G$ we have
 \begin{align*}
  M(u_1,v_1,u_2,v_2)^\top =\, &m_1u_1+m_2v_1+m_3u_2+m_4v_2\\
  \leq\, &u_1(m_1+m_2)+m_2+u_2(m_3+m_4)+m_4\\
  =\, &m_2+m_4 < p.
 \end{align*}
 Hence $(u_1,v_1)$ and $(u_2,v_2)$ are not conflicting and $M_p(G)$ is an empty graph on $|E(G)| \geq \binom{k}{2}$ edges.
 Since equality is attained for $G = K_k$, analogously to the first item $A(M,p,k) = \binom{k}{2}$ and $W(M,p,k) = 1$.

 \item
 We assume that $m_2+m_4 > 0$ or ($m_2+m_4 = 0$ and $m_1,m_2 \neq 0$).
 Fix some integer $q \geq \max\{2,p\}$ and let $V=\{q^i\mid i\in\bZ,i\geq 1\}$.
 If $m_2 + m_4 = 0$ and $m_1,m_2 \neq 0$, then additionally ensure that $q$ is not a factor of $m_1$. 
 We claim that if $V(G) \subset V$, then $M_p(G)$ is a complete graph.
 Indeed, consider two edges $(q^i,q^s)$, $(q^j,q^t)$, $i<s$, $j<t$ and $i\leq j$.
 If $m_2 + m_4 > 0$, then
 \begin{align*}
  &\hspace{-3em} M(q^i,q^s,q^j,q^t)^\top + M(q^j,q^t,q^i,q^s)^\top\\
  =\, &m_1q^i+m_2q^s+m_3q^j+m_4q^t+m_1q^j+m_2q^t+m_3q^i+m_4q^s\\
  =\, &(m_1+m_3)(q^i+q^j)+(m_2+m_4)(q^s+q^t)\\
  \geq\, &(m_1+m_3)(q^i+q^j)+(m_2+m_4)(q^{i+1}+q^{j+1})\\
  \geq\, &(m_1+m_3)(q^i+q^j)+(m_2+m_4)(q^i+q^j+2q)\\
  \geq\, &\underbrace{(m_1+m_2+m_3+m_4)}_{=0}(q^i+q^j)+(m_2+m_4)2q\\
  \geq\, &2q \geq 2p.
 \end{align*}
 Thus $M(q^i,q^s,q^j,q^t)^\top\geq p$ or $M(q^j,q^t,q^i,q^s)^\top\geq p$.
  
 If $m_2+m_4=0$ and $m_1,m_2 \neq 0$, then $m_1=-m_3$ and $m_2=-m_4$, as $M$ is translation invariant.
 Recall that $q$ divides neither $m_1$ nor $m_2$.
 Further recall that $i\leq j,s,t$ and $i<s,t$.
 We have
 \begin{align*}
  |M(q^i,q^s,q^j,q^t)^\top | &= |m_1q^i+m_2q^s+m_3q^j+m_4q^t|\\
  &= |m_1(q^i-q^j)+m_2(q^s-q^t)|\\
  & = q^i\, |m_1(1-q^{j-i})+m_2(q^{s-i}-q^{t-i})|\\
  & \geq q^i \geq p.  
 \end{align*}
 Note that $m_1(1-q^{j-i})+m_2(q^{s-i}-q^{t-i})\neq 0$, since $i-j=0$ implies $s\neq t$ (as edges $q^iq^s$ and $q^jq^t$ are distinct) and since $q$ is a factor of $m_2(q^{s-i}-q^{t-i})$ but not of $m_1(1-q^{j-i})$.
 Hence $M(q^i,q^s,q^j,q^t)^\top \geq p$ or $ M(q^j,q^t,q^i,q^s)^\top = m_1q^j+m_2q^t+m_3q^i+m_4q^s = -m_3q^j-m_4q^t-m_1q^i-m_2q^s = - M(q^i,q^s,q^j,q^t)^\top \geq p$.
 
 This shows that any two edges with endpoints in $V$ are in conflict.
 Therefore $M_p(G)$ is a complete graph if $V(G)\subset V$.
 This shows that $A(M,p,k) = 1$.

 \item
 We assume that $m_2,m_4 < 0$.
 Fix some integer $q \geq \max\{2,-p,(1-m_1)/m_2, (1-m_3)/m_4\}$ and let $V=\{q^i\mid i\in\bZ,i\geq 1\}$.
 We claim that if $V(G)\subset V$, then $M_p(G)$ is an empty graph.
 Indeed, for any two edges $(q^i,q^j)$ and $(q^s,q^t)$, $i < j$, $s < t$, we have
 \begin{align*}
  M(q^i,q^j,q^s,q^t)^\top =\, &m_1q^i+m_2q^j+m_3q^s+m_4q^t\\
  =\, &q^i(m_1+m_2q^{j-i}) + q^s(m_3+m_4q^{t-s})\\
  \leq\, &q^i(\underbrace{m_1+m_2q}_{\leq -1}) + q^s(\underbrace{m_3+m_4q}_{\leq -1})\\
  \leq\, &-q^i-q^s \leq -q-1 \leq p-1.
 \end{align*}
 This shows that no edges with vertices in $V$ are in conflict.
 Therefore $M_p(G)$ is an empty graph if $V(G)\subset V$.
 Similarly to above we have $W(M,p,k) = 1$\qed
 \end{enumerate}

\section{Reductions and Preliminary Lemmas}\label{sec:prelim}

This section contains some preliminary lemmas preparing the proof of Theorem~\ref{thm:specialMatrices} in Section~\ref{sec:proofs}.
Let $M \in \bZ^{s \times 4}$ and $p,k \in \bZ$ with $k \geq 2$.
We start with some basic operations on $M$ and $p$ and their effect on $A(M,p,k)$ and $W(M,p,k)$.
The first such operation follows immediately from the definition of conflicting edges.

\begin{observation}\label{obs:swapping-columns}
 Swapping the first column in $M$ with the third and the second with the fourth preserves all conflicts and non-conflicts.
 Hence if $M'$ denotes the resulting matrix, we have $A(M',p,k) = A(M,p,k)$ and $W(M',p,k) = W(M,p,k)$.
\end{observation}

The next lemma provides another such operation, as well as an operation on $1 \times 4$--matrices that exchanges the roles of conflicts for non-conflicts.
(Later in Section~\ref{sec:nest} we prove a similar result for one specific $2 \times 4$--matrix.)
Let us remark that for some $M$ and $p$, for example for $M = \MM{-}{+}{0}{0}{-}{0}{+}{0}$ and $p \geq 3$, there is no matrix $\tilde{M}$ and integer $\tilde{p}$ such that for every ordered graph $G$ we have that $M_p(G)$ is the complement of $\tilde{M}_{\tilde{p}}(G)$.

For a matrix $M$ let $-M$ be the matrix obtained by multiplying every entry by $-1$, and $\overline{M}$ be the matrix obtained from $M$ by reversing the order of its columns.

\begin{lemma}\label{lem:small-facts}
 For every matrix $M \in \bZ^{s\times 4}$ and all integers $p,k \in \bZ$, $k \geq 2$ each of the following holds.
 \begin{enumerate}[label = (\roman*)]
  \item $A(M,p,k) = A(-\overline{M},p,k)$ and $W(M,p,k) = W(-\overline{M},p,k)$.\label{enum:reverse-matrix}
  
  \item If $M = \M{m_1~}{m_2~}{m_3~}{m_4}$, then
   \[
    A(M,p,k) = W(\MM{-m_1}{-m_2}{-m_3}{-m_4}{-m_3}{-m_4}{-m_1}{-m_2},1-p,k)
   \]
   and\label{enum:exchange}
   \[
    W(M,p,k) = A(\MM{-m_1}{-m_2}{-m_3}{-m_4}{-m_3}{-m_4}{-m_1}{-m_2},1-p,k).
   \]
 \end{enumerate}
\end{lemma}
\begin{proof}
 \begin{enumerate}[label = (\roman*)]
  \item Consider for any ordered graph $G$ the ordered graph $-G$ obtained by multiplying every vertex position by $-1$ (so $(-x,-y)$ is an edge in $-G$ if and only if $(y,x)$ is an edge in $G$).
  For an edge $e$ in $G$ let $e^-$ be the corresponding edge in $-G$.
   Intuitively speaking, $-G$ is obtained from $G$ by exchanging the meanings of left and right.
   Now for any two edges $e_1,e_2$ in $G$, say $e_1 = (u_1,v_1)$ and $e_2 = (u_2,v_2)$, we have
   \begin{align*}
    & \hspace{-3em} e_1e_2 \in E(M_p(G)) \\
    \Leftrightarrow  \hspace{2em} &M(u_1,v_1,u_2,v_2)^\top \geq \mathbf{p} \lor M(u_2,v_2,u_1,v_1)^\top \geq \mathbf{p}\\
    \Leftrightarrow  \hspace{2em} &(-M)(-u_1,-v_1,-u_2,-v_2)^\top \geq \mathbf{p}\\
     \lor \hspace{0.5em} &(-M)(-u_2,-v_2,-u_1,-v_1)^\top \geq \mathbf{p}\\
    \Leftrightarrow  \hspace{2em} &(-\overline{M})(-v_2,-u_2,-v_1,-u_1)^\top \geq \mathbf{p}\\
     \lor \hspace{0.5em} &(-\overline{M})(-v_1,-u_1,-v_2,-u_2)^\top \geq \mathbf{p}\\
    \Leftrightarrow  \hspace{2em} &e^-_1e^-_2\in E(-\overline{M}_p(-G)).
   \end{align*}
   Thus mapping an edge $e$ from $G$ to $e^-$ yields an isomorphism between $M_p(G)$ and $-\overline{M}_p(-G)$.
   Since $\chi(G)=\chi(-G)$ we get $A(M,p,k) = A(-\overline{M},p,k)$ and $W(M,p,k) = W(-\overline{M},p,k)$.
  
   \item Let $e_1,e_2$ be any two edges in a given ordered graph $G$.
   Say $e_1 = (u_1,v_1)$ and $e_2 = (u_2,v_2)$.
   Then for $M' = \MM{-m_1}{-m_2}{-m_3}{-m_4}{-m_3}{-m_4}{-m_1}{-m_2}$ we have 
   \begingroup\allowdisplaybreaks
   \begin{align*}
    & \hspace{-3em} e_1e_2 \in E(M_p(G)) \\
    \Leftrightarrow \hspace{2em} &m_1u_1 + m_2v_1 + m_3u_2 + m_4v_2 \geq p\\
    \lor \hspace{0.5em} &m_1u_2 + m_2v_2 + m_3u_1 + m_4v_1 \geq p\\
    \Leftrightarrow \hspace{0.5em} \lnot \hspace{0.5em} &\big(m_1u_1 + m_2v_1 + m_3u_2 + m_4v_2 \leq p-1\\
     &\land m_1u_2 + m_2v_2 + m_3u_1 + m_4v_1 \leq p-1 \big)\\
    \Leftrightarrow\hspace{0.5em} \lnot \hspace{0.5em} & \big(-m_1u_1 - m_2v_1 - m_3u_2 - m_4v_2 \geq 1-p\\
    & \land -m_1u_2 - m_2v_2 - m_3u_1 - m_4v_1 \geq 1-p \big)\\
    \Leftrightarrow\hspace{0.5em} \lnot \hspace{0.5em} &\MM{-m_1}{-m_2}{-m_3}{-m_4}{-m_3}{-m_4}{-m_1}{-m_2}(u_1,v_1,u_2,v_2)^\top \geq \mathbf{1-p}\\
    \Leftrightarrow\hspace{0.5em} \lnot \hspace{0.5em} &\Big(\MM{-m_1}{-m_2}{-m_3}{-m_4}{-m_3}{-m_4}{-m_1}{-m_2}(u_1,v_1,u_2,v_2)^\top \geq \mathbf{1-p}\\
    & \lor \MM{-m_1}{-m_2}{-m_3}{-m_4}{-m_3}{-m_4}{-m_1}{-m_2}(u_2,v_2,u_1,v_1)^\top \geq \mathbf{1-p}\Big)\\
    \Leftrightarrow\hspace{2em} &e_1e_2 \not\in E(M'_{1-p}(G)).
   \end{align*}%
   \endgroup
   Therefore $M_p(G)$ is the complement of $M'_{1-p}(G)$.
   From this we get $A(M,p,k) = W(M',1-p,k)$ and $W(M,p,k) = A(M',1-p,k)$.
 \end{enumerate}
 \vspace{-\baselineskip}
\end{proof}

We close this section with a lemma, which is needed in some of the proofs in Section~\ref{sec:proofs}.

\begin{lemma}\label{lem:LongEdges}
  Let $k$, $q$ denote positive integers with $k > q$, and let $G$ be any ordered graph.
  If $\chi(G)\geq k$, then there is a set $S$ of $\binom{k-q+1}{2}$ edges of length at least $q$ in $G$.
  Moreover, if $k\geq 3$, then there is an edge of length at least $q-1$ in $G$ that is not in $S$.
\end{lemma}
\begin{proof}
  Let $G'$ denote a $k$-critical subgraph of $G$ with vertices $v_1<\cdots<v_t$, for some $t\geq k$.
  Then $G'$ has minimum degree $k-1$.
  We have that $v_i$ is a right endpoint of at most $i-1$ edges in $G'$ and a left endpoint of at most $q-1$ edges of length at most $q-1$ in $G'$, $i=1,\ldots,k-q$.
  Note that $k-q\geq 1$.
  Hence $v_i$ is left endpoint of at least $k-1-(i-1)-(q-1)=k-i-q+1$ edges of length at least $q$.
  Thus there is a set $S$ with $\sum_{i=1}^{k-q}k-q-i+1=\binom{k-q+1}{2}$ edges of length at least $q$ in $G'$.
  Moreover, if $k\geq 3$, then $v_1$ is either incident to another edge of length at least $q$ (which is not in $S$) or to an edge of length $q-1$.
\end{proof}

\section{Proof of Theorem~\ref{thm:specialMatrices}}\label{sec:proofs}

We prove Theorem~\ref{thm:specialMatrices} by considering all $19$ translation invariant matrices $M \in \{-1,0,1\}^{1 \times 4}$.
Observation~\ref{obs:swapping-columns} and Lemma~\ref{lem:small-facts}~\ref{enum:reverse-matrix} allow us to group these into ten groups of equivalent matrices, corresponding to rows $2$ to $11$ in Table~\ref{tab:detailed}, and consider only one representative matrix per group (marked with $^\star$ in the table).
The first case, $M = \M{0}{0}{0}{0}$ corresponding to row~$2$ in Table~\ref{tab:detailed}, can be completely handled with Theorem~\ref{thm:generalMatrices}.
If $p \leq 0$, then by Theorem~\ref{thm:generalMatrices}~\ref{enum:m2m40} for all $k \geq 2$ we have $A(M,p,k) = \alpha(M_p(K_k)) = 1$ for some ordered graph $K_k$ and $W(M,p,k) = \omega(M_p(K_k)) = \binom{k}{2}$ for some ordered graph $K_k$.
And if $p > 0$, then by Theorem~\ref{thm:generalMatrices}~\ref{enum:m1m20} for all $k \geq 2$ we have $A(M,p,k) = \alpha(M_p(K_k)) = \binom{k}{2}$ for some ordered graph $K_k$ and $W(M,p,k) = \omega(M_p(K_k)) = 1$ for some ordered graph $K_k$.

The remaining $18$ translation invariant matrices in $\{-1,0,1\}^{1 \times 4}$ come in nine groups corresponding to rows $3$ to $11$ in Table~\ref{tab:detailed} and are handled in Propositions~\ref{P0M0} -- \ref{M00P} below.
Let us emphasize that in all cases our upper bounds on $A(M,p,k)$ and $W(M,p,k)$ are attained by some ordered graph $K_k$.


\begin{proposition}[left endpoints at distance at least $p$, row $3$ in Table~\ref{tab:detailed}]\label{P0M0}
 ~\\
 Let $M \in \{\M{+}{0}{-}{0}, \M{-}{0}{+}{0}, \M{0}{+}{0}{-}, \M{0}{-}{0}{+}\}$.
 
 \begin{compactitem}[\enskip]
  \item If $p \leq 0$, then $A(M,p,k) = 1$ and $W(M,p,k) = \binom{k}{2}$ for all $k \geq 2$.
  
  \item If $p \geq 1$, then $A(M,p,k) = k-1$ and $W(M,p,k) = \ceil{\frac{k-1}{p}}$ for all $k \geq 2$.
 \end{compactitem} 
\end{proposition}
\begin{proof}

 Consider $M=\M{+}{0}{-}{0}$.
 If $p\leq 0$, then $A(M,p,k) = 1$ and $W(M,p,k) = \binom{k}{2}$ for all $k\geq 2$ due to Theorem~\ref{thm:generalMatrices}~\ref{enum:m2m40}.
 This leaves to consider $A(M,p,k)$ and $W(M,p,k)$ in the case $p\geq 1$.

 Here two edges $e_1, e_2$ are conflicting if and only if their left endpoints differ by at least $p$. 
 A clique in $M_p(G)$ is a set of edges in $G$ whose left endpoints are pairwise at distance at least $p$.
 An independent set in $M_p(G)$ is a set of edges in $G$ whose left endpoints are pairwise at distance at most $p-1$, i.e., all left endpoints are contained in some closed interval of length at most $p-1$.

 \medskip
 
 Consider $A(M,p,k)$ for $p \geq 1$.
 For any $k \geq 2$ consider $G = K_k$ with vertex set $V = \{ip \mid i \in [k]\}$.
 As any two vertices in $V$ have distance at least $p$, two edges are non-conflicting if and only if their left endpoints coincide.
 Thus we have $A(M,p,k) \leq \alpha(M_p(G)) = k-1$, as certified by the edges incident to the leftmost vertex.
 
 \medskip
 
 Now consider any ordered graph $G$ with $\chi(G) = k$.
 Let $G'$ denote a $k$-critical subgraph of $G$, i.e., $G'$ has minimum degree at least $k-1$.
 Then $M_p(G')$ is an induced subgraph of $M_p(G)$ and hence $\alpha(M_p(G)) \geq \alpha(M_p(G'))$.
 For a vertex $v$ in $G'$ the set of all edges with left endpoint $v$ forms an independent set in $M_p(G')$.
 In particular the leftmost vertex in $G'$ is left endpoint of at least $k-1$ edges.
 Hence $\alpha(M_p(G)) \geq \alpha(M_p(G'))\geq k-1$.
 As $G$ was arbitrary, this shows that $A(M,p,k) \geq k-1$.
 
 \medskip
 
 Consider $W(M,p,k)$ for $p \geq 1$.
 For any $k \geq 2$ consider $G=K_k$ with vertex set $[k]$.
 Recall that two edges are conflicting if their left endpoints differ by at least $p$.
 Clearly, a largest clique in $M_p(G)$ is formed by considering every $p^\text{th}$ vertex of $G$ and taking one edge with this as its left endpoint.
 It follows that $W(M,p,k) \leq \omega(M_p(G)) = \ceil{(k-1)/p}$.
 
 \medskip
 
 Now consider any ordered graph $G$ with $\chi(G) = k$.
 Let $G'$ denote a $k$-critical subgraph of $G$, i.e., $G'$ has minimum degree at least $k-1$.
 Then $M_p(G')$ is a subgraph of $M_p(G)$ and hence $\omega(M_p(G)) \geq \omega(M_p(G'))$.
 Consider the set $F$ that consists of every $p^\text{th}$ edge incident to the rightmost vertex in $G'$.
 Then $F$ forms a clique in $M_p(G')$ and hence $\omega(M_p(G'))\geq \ceil{(k-1)/p}$.
 It follows that $W(M,p,k) \geq \omega(M_p(G)) \geq \ceil{(k-1)/p}$, since $G$ was arbitrary.
 
 \medskip
 
 Finally, if $M' \in \{\M{-}{0}{+}{0}, \M{0}{+}{0}{-}, \M{0}{-}{0}{+}\}$, then $M'$ is obtained from $M=\M{+}{0}{-}{0}$ either by switching the first with the third and the second with the last column or by reversing the order of columns in $-M$ or both.
 Thus $W(M',p,k)=W(M,p,k)$ and $A(M',p,k)=A(M,p,k)$, due to Observation~\ref{obs:swapping-columns} and Lemma~\ref{lem:small-facts} \ref{enum:reverse-matrix}.
%
%
%
%
\end{proof}

%
%
%
%
%
%
%


\begin{proposition}[at least one edge of length at least $p$, row $4$ in Table~\ref{tab:detailed}]\label{MP00}
 ~\\
 Let $M \in\{\M{-}{+}{0}{0},\M{0}{0}{-}{+}\}$.
 \begin{compactitem}[\enskip]
  \item If $p \leq 1$, then $A(M,p,k) = 1$ and $W(M,p,k) = \binom{k}{2}$ for all $k \geq 2$.
 
  \item If $p \geq 2$, then $A(M,p,k) = 1$ for all $k \geq 2$ and $W(M,p,k)=1$ for $2\leq k\leq p$ and $W(M,p,k) = 1+\binom{k-p+1}{2}$ for $k \geq p+1$.
 \end{compactitem} 
\end{proposition}
\begin{proof}
 
 Consider $M=\M{-}{+}{0}{0}$.
 We have $A(M,p,k) = 1$ for all $p,k \in \bZ$, $k \geq 2$ and, if $p\leq 1$, $W(M,p,k) = \binom{k}{2}$ for all $k\geq 2$ due to Theorem~\ref{thm:generalMatrices}~\ref{enum:m2m40} and~\ref{enum:m2m4greater0}.
 This leaves to consider $W(M,p,k)$ in the case $p\geq 2$.

 Here two edges $e_1,e_2$ are conflicting if and only if at least one of them has length at least $p$.
 We call edges of length at least $p$ the \emph{long edges} and edges of length at most $p-1$ the \emph{short edges}.
 A clique in $M_p(G)$ is a set of edges in $G$, at most one of which is short, while an independent set in $M_p(G)$ is a set of edges in $G$ with only short edges.
 Hence $\omega(M_p(G))$ is just the total number of long edges (plus one if there is at least one short edge), and $\alpha(M_p(G))$ is the number of short edges.

 \medskip

 For any $k \geq 2$ consider $G=K_k$ with a vertex set $[k]$.
 If $k\leq p$, then there are no long edges.
 Hence $W(M,p,k) \leq \omega(M_p(G))=1$, since $\chi(G) = k$.
 Therefore $W(M,p,k)=1$.
 If $k\geq p+1$, then there is at least one long and one short edge in $G$.
 There are $k-\ell$ edges of length $\ell$ in this $G$.
 In particular, since $\chi(G) = k$, $W(M,p,k) \leq \omega(M_p(G)) = 1+\sum_{\ell=p}^{k-1}(k-\ell)=1+\binom{k-p+1}{2}$.
 
 \medskip
 
 Now consider an arbitrary ordered graph $G$ with $\chi(G)= k \geq p+1$.
 By Lemma~\ref{lem:LongEdges} there are $1+\binom{k-p+1}{2}$ edges such that all but one of them is long.
 Hence $\omega(M_p(G)) \geq 1+\binom{k-p+1}{2}$.
 This shows that $W(M,p,k)=1+\binom{k-p+1}{2}$.
 
 \medskip
 
 Finally, if $M' = \M{0}{0}{-}{+}$, then $M'$ is obtained from $M=\M{-}{+}{0}{0}$ by switching the first with the third and the second with the last column.
 Thus $W(M',p,k)=W(M,p,k)$ and $A(M',p,k)=A(M,p,k)$.
\end{proof}

%
%
%
%
%
%
%
%
%


\begin{proposition}[At least one edge of length at most $-p$, row $5$ in Table~\ref{tab:detailed}]\label{PM00}
 ~\\
 Let $M \in \{\M{+}{-}{0}{0}, \M{0}{0}{+}{-}\}$.
 \begin{compactitem}[\enskip]
  \item If $p \geq 0$, then $A(M,p,k) = \binom{k}{2}$ and $W(M,p,k) = 1$ for all $k \geq 2$.
 
  \item If $p \leq -1$, then $A(M,p,k) = 1$ for $2\leq k \leq |p|+1$ and $A(M,p,k) = \binom{k-|p|}{2}$ for $k \geq |p|+2$ and $W(M,p,k) = 1$ for all $k \geq 2$.
 \end{compactitem}
\end{proposition}
\begin{proof}
 
 consider $M=\M{+}{-}{0}{0}$.
 We have $W(M,p,k) = 1$ for all $p,k \in \bZ$, $k \geq 2$ and, if $p\geq 0$, $A(M,p,k) = \binom{k}{2}$ for all $k\geq 2$ due to Theorem~\ref{thm:generalMatrices}~\ref{enum:m1m20} and~\ref{enum:m2m4smaller0}.
 This leaves to consider $A(M,p,k)$ in the case $p \leq -1$.

 Here edges $(u_1,v_1)$ and $(u_2,v_2)$ are conflicting if $v_1-u_1\leq -p$ or $v_2-u_2 \leq -p$, that is, if one of the edges has length at most $-p$.
 Let $q=-p$.
 We call edges of length at least $q+1$ the \emph{long edges} and edges of length at most $q$ the \emph{short edges}.
 Then a clique in $M_p(G)$ is a set of edges in $G$, at most one of which is long, while an independent set in $M_p(G)$ is a set of edges in $G$ with only long edges.
 Hence $\omega(M_p(G))$ is just the total number of short edges (plus one if there is at least one long edge), and $\alpha(M_p(G))$ is the total number of long edges.

 \medskip

 For any $k \geq 2$ consider $G=K_k$ with vertex set $[k]$.
 If $k\leq q+1$, then there are no long edges.
 Hence $A(M,p,k) \leq \alpha(M_p(G))=1$, since $\chi(G) = k$.
 Therefore $A(M,p,k)=1$.
 If $k\geq q+2$, then there is at least one long edge in $G$.
 There are $k-\ell$ edges of length $\ell$ in this $G$.
 In particular, since $\chi(G) = k$, $A(M,p,k) \leq \alpha(M_p(G)) = \sum_{\ell=q+1}^{k-1}(k-\ell)=\binom{k-|p|}{2}$.
 
 \medskip
 
 Now consider an arbitrary ordered graph $G$ with $\chi(G)= k \geq q+2$.
 Then there are at least $\binom{k-q}{2}$ long edges in $G$ due to Lemma~\ref{lem:LongEdges}, i.e., $\alpha(M_p(G)) \geq \binom{k-q}{2}$.
 This shows that $A(M,p,k)=\binom{k-q}{2}$.
 
 \medskip
 
 Finally, if $M' = \M{0}{0}{+}{-}$, then $M'$ is obtained from $M=\M{+}{-}{0}{0}$ by switching the first with the third and the second with the last column.
 Thus $W(M',p,k)=W(M,p,k)$ and $A(M',p,k)=A(M,p,k)$.
\end{proof}

%
%
%
%
%
%
%
%
%
%
%
%
%
%


\begin{proposition}[lengths sum to at least $p$, row $6$ in Table~\ref{tab:detailed}]\label{MPMP}
 ~\\
 Let $M = \M{-}{+}{-}{+}$.
 \begin{compactitem}[\enskip]
  \item If $p \leq 2$, then $A(M,p,k) = 1$ and $W(M,p,k) = \binom{k}{2}$ for all $k \geq 2$.
 
  \item If $p \geq 3$, then $A(M,p,k) = 1$ for all $k \geq 2$ and $W(M,p,k) = 1$ for $2 \leq k \leq \ceil{\frac{p}{2}}$ and $W(M,p,k) = p \bmod 2 + \binom{k-\ceil{\frac{p}{2}}+1}{2}$ for $k \geq \ceil{\frac{p}{2}}+1$.
 \end{compactitem}
\end{proposition}
\begin{proof}

 We have $A(M,p,k) = 1$ for all $p,k \in \bZ$, $k \geq 2$ and, if $p\leq 2$, $W(M,p,k) = \binom{k}{2}$ for all $k\geq 2$ due to Theorem~\ref{thm:generalMatrices}~\ref{enum:m2m40} and~\ref{enum:m2m4greater0}.
 This leaves to consider $W(M,p,k)$ in the case $p\geq 3$.

 Here edges $(u_1,v_1)$ and $(u_2,v_2)$ are conflicting if $v_1-u_1 + v_2-u_2 \geq p$, that is, if their lengths add up to at least $p$.
 Let $q = \ceil{p/2}-1$.
 We call an edge \emph{short} if its length is at most $q$, and \emph{long} otherwise.
 Then a clique in $M_p(G)$ could be of two kinds.
 Either it is a set of only long edges in $G$, or there is one short edge of length $\ell \leq q$ and each remaining edge has length at least $p - \ell$.
 An independent set in $M_p(G)$ is a set of edges in $G$ where the lengths of any two longest edges add up to at most $p-1$.

 \medskip

 First consider $G = K_k$ with vertex set $[k]$.
 If $k\leq \ceil{p/2}$, then the largest sum of the lengths of two edges in $G$ is $k-1+k-2\leq p-1$.
 Hence $M_p(G)$ is empty and, since $\chi(G)=k$, $W(M,p,k) \leq \omega(M_p(G))=1$.
 Therefore $W(M,p,k)=1$ in this case.
 Now consider $k \geq \ceil{p/2} + 1$. 
 Recall that for each $\ell = 1,\ldots,k-1$ there are exactly $k-\ell$ edges of length exactly $\ell$ in $G$.
 A largest clique in $M_p(G)$ contains all long edges and, if $p$ is odd, one edge of length $q$.
 It follows that if $p$ is even, then $\omega(M_p(K_k)) = \sum_{\ell=q+1}^{k-1}(k-\ell) = \binom{k-q}{2}$.
 While if $p$ is odd, then $\omega(M_p(K_k)) = 1 + \sum_{\ell = q+1}^{k-1} (k-\ell) = 1+\binom{k-q}{2}$.
 Altogether this shows that $W(M,p,k) \leq p \bmod 2 + \binom{k-q}{2}$.
 
 \medskip
 
 Now consider an arbitrary ordered graph $G$ with $\chi(G)=k \geq \ceil{p/2}+1 \geq 3$.
 By Lemma~\ref{lem:LongEdges} there is a set $S$ with $1+\binom{k-q}{2}$ edges, such that one of them, say $e$, has length at least $q$ and all others are long.
 If $p$ is odd, then $S$ is a clique in $M_p(G)$.
 If $p$ is even, then $S - e$ is a clique in $M_p(G)$. 
 Therefore $\omega(M_p(G)) \geq \binom{k-q}{2} + p \bmod 2$.
 
 Hence $W(M,p,k) \geq p \bmod 2 + \binom{k-q}{2}$.
 Altogether $W(M,p,k) = p \bmod 2 + \binom{k-\ceil{p/2}+1}{2}$, if $k \geq \ceil{p/2}+1$, and $W(M,p,k) = 1$ otherwise.
\end{proof}

%
%
%
%
%
%

\begin{proposition}[lengths sum to at most $p$, row $7$ in Table~\ref{tab:detailed}]\label{PMPM}
 ~\\
 Let $M = \M{+}{-}{+}{-}$.
 \begin{compactitem}[\enskip]
  \item If $p \leq -1$, then $A(M,p,k) = 1$ for $2 \leq k \leq \ceil{\frac{1-p}{2}}$ and $A(M,p,k) = (1-p) \bmod 2 + \binom{k-\ceil{\frac{1-p}{2}}+1}{2}$ for $k \geq \ceil{\frac{1-p}{2}}+1$ and $W(M,p,k) = 1$ for all $k \geq 2$.
  
  \item If $p \geq -1$, then $A(M,p,k) = \binom{k}{2}$ and $W(M,p,k) = 1$ for all $k \geq 2$.  
 \end{compactitem}
\end{proposition}
\begin{proof}
 This follows immediately from Proposition~\ref{MPMP} and Lemma~\ref{lem:small-facts}~\ref{enum:exchange}.
\end{proof}


\begin{proposition}[lengths differ by at least $p$, row $8$ in Table~\ref{tab:detailed}]\label{PMMP}
 ~\\
 Let $M \in \{\M{+}{-}{-}{+}, \M{-}{+}{+}{-}\}$.
 \begin{compactitem}[\enskip]
  \item If $p \leq 0$, then $A(M,p,k) = 1$ and $W(M,p,k) = \binom{k}{2}$ for all $k \geq 2$.
 
  \item If $p \geq 1$, then $A(M,p,k) = 1$ and $W(M,p,k) = \ceil{\frac{k-1}{p}}$ for all $k \geq 2$.
 \end{compactitem}
\end{proposition}
\begin{proof}
 
 Consider $M=\M{+}{-}{-}{+}$.
 We have $A(M,p,k) = 1$ for all $p,k \in \bZ$, $k \geq 2$ and, if $p\leq 0$, $W(M,p,k) = \binom{k}{2}$ for all $k\geq 2$ due to Theorem~\ref{thm:generalMatrices}~\ref{enum:m2m40} and~\ref{enum:m2m4greater0}.
 This leaves to consider $W(M,p,k)$ in the case $p\geq 1$.

 Here edges $(u_1,v_1)$ and $(u_2,v_2)$ are conflicting if $(v_2-u_2)-(v_1-u_1) \geq p$ or $(v_1-u_1)-(v_2-u_2) \geq p$, i.e., if their lengths differ by at least $p$.
 A clique in $M_p(G)$ is a set of edges in $G$ whose lengths differ pairwise by at least $p$.
 An independent set in $M_p(G)$ is a set of edges in $G$ whose lengths differ pairwise by at most $p-1$, i.e., all lengths are contained in some closed interval of length at most $p-1$.

\medskip

 Consider $G = K_k$ with vertex set $[k]$.
 The edges of $K_k$ determine exactly $k-1$ different lengths $1,\ldots,k-1$.
 As cliques in $M_p(G)$ correspond to edge sets in $G$ with lengths pairwise differing by at least $p$, a maximum clique in $M_p(G)$ has size $\ceil{(k-1)/p}$.
 Thus we have $W(M,p,k) \leq \omega(M_p(G)) = \ceil{(k-1)/p}$, as desired.
 
 \medskip
 
 Now consider an arbitrary ordered graph $G$ with $\chi(G)=k$.
 Let $G'$ denote a $k$-critical subgraph of $G$.
 Then $M_p(G')$ is a subgraph of $M_p(G)$ and hence $\omega(M_p(G))\geq \omega(M_p(G'))$.
 Let $F$ be the set of edges incident to the leftmost vertex $v$ in $G$.
 All edges in $F$ have pairwise distinct lengths, i.e., taking the subset of $F$ corresponding to every $p^{\text{th}}$ length gives a clique in $M_p(G)$.
 Hence $\omega(M_p(G')) \geq \ceil{(k-1)/p}$, since $G'$ has minimum degree at least $k-1$.
 As $G$ was arbitrary, this gives $W(M,p,k) \geq \ceil{(k-1)/p}$.
 
 \medskip
 
 Finally, if $M' = \M{-}{+}{+}{-}$, then $M'$ is obtained from $M=\M{+}{-}{-}{+}$ by switching the first with the third and the second with the last column.
 Thus $W(M',p,k)=W(M,p,k)$ and $A(M',p,k)=A(M,p,k)$.
\end{proof}

%
%
%
%
%
%
%
%
%
%
%
%
%


\begin{proposition}[midpoints are at distance at least $p/2$, row $9$ in Table~\ref{tab:detailed}]\label{PPMM}
 ~\\
 Let $M \in \{\M{+}{+}{-}{-}, \M{-}{-}{+}{+}\}$.
 \begin{compactitem}[\enskip]
  \item If $p \leq 0$, then $A(M,p,k) = 1$ and $W(M,p,k) = \binom{k}{2}$ for all $k \geq 2$.
 
  \item If $p \geq 1$, then $A(M,p,k) = 1$ and $W(M,p,k) = \ceil{\frac{2k-3}{p}}$ for all $k \geq 2$.
 \end{compactitem}
\end{proposition}
\begin{proof}
 
 Consider $M=\M{+}{+}{-}{-}$.
 We have $A(M,p,k) = 1$ for all $p,k \in \bZ$, $k \geq 2$ and, if $p\leq 0$, $W(M,p,k) = \binom{k}{2}$ for all $k\geq 2$ due to Theorem~\ref{thm:generalMatrices}~\ref{enum:m2m40} and~\ref{enum:m2m4greater0}.
 This leaves to consider $W(M,p,k)$ in the case $p\geq 1$.

 For an edge $(u_1,v_1)$ we think of $(u_1 + v_1)/2$ as its \emph{midpoint}.
 Note that the midpoints are not necessarily integers.
 Then edges $(u_1,v_1)$ and $(u_2,v_2)$ are conflicting if $|(u_1+v_1)/2 - (u_2+v_2)/2| \geq p/2$, that is, if their midpoints are of distance at least $p/2$.
 A clique in $M_p(G)$ is a set of edges in $G$ whose midpoints are pairwise at distance at least $p/2$.
 An independent set in $M_p(G)$ is a set of edges in $G$ whose midpoints are pairwise at distance at most $(p-1)/2$, i.e., all midpoints are contained in some closed interval of length at most $(p-1)/2$.

\medskip

 Consider $G = K_k$ with vertex set $[k]$.
 The edges of $G$ determine exactly $2k-3$ midpoints; one for each vertex that is neither the first nor the last vertex, and one for each gap between two consecutive vertices.
 Consecutive midpoints are at distance $1/2$.
 As cliques in $M_p(G)$ correspond to edge sets in $G$ with midpoints at pairwise distance at least $p/2$, a maximum clique in $M_p(G)$ has size $\ceil{\frac{(2k-3)/2}{p/2}} = \ceil{(2k-3)/p}$.
 Thus we have $W(M,p,k) \leq \omega(M_p(G)) = \ceil{(2k-3)/p}$, as desired.

 \medskip
 
 Now consider an arbitrary ordered graph $G$ with $\chi(G)=k$.
 Let $G'$ denote a $k$-critical subgraph of $G$.  
 Then $M_p(G')$ is a subgraph of $M_p(G)$ and hence $\omega(M_p(G))\geq \omega(M_p(G'))$. 
 Let $F$ be the set of edges incident to the first vertex or the last vertex (or both).
 All edges in $F$ have pairwise distinct midpoints.
 Taking a subset of $F$ corresponding to every $p^{\text{th}}$ midpoint of edges in $F$ gives a clique in $M_p(G')$.
 Hence $\omega(M_p(G')) \geq \ceil{|F|/p} \geq \ceil{(2k-3)/p}$, since $G'$ has minimum degree $k-1$.
 As $G$ was arbitrary, this gives $W(M,p,k) \geq \ceil{(2k-3)/p}$.
 Altogether $W(M,p,k) = \ceil{(2k-3)/p}$.
 
 \medskip
 
 Finally, if $M' = \M{-}{-}{+}{+}$, then $M'$ is obtained from $M=\M{+}{+}{-}{-}$ by switching the first with the third and the second with the last column.
 Thus $W(M',p,k)=W(M,p,k)$ and $A(M',p,k)=A(M,p,k)$ due to Observation~\ref{obs:swapping-columns}.
\end{proof}

%
%
%
%
%
%
%
%
%
%
%
%
 

For the next matrix $M = \M{+}{0}{0}{-}$ we can not rely on Theorem~\ref{thm:generalMatrices}.
As for the previous matrices, there are four cases to be considered: $A(M,p,k)$ and $W(M,p,k)$ for $p \leq 0$ and $p \geq 1$.
However, before we determine $A(M,p,k)$ and $W(M,p,k)$, we first investigate the structure of independent sets in $M_p(G)$ for an ordered graph $G$ and prove three lemmas.

A \emph{comparability graph} is a graph admitting a transitive orientation of its edges, i.e., an orientation such that for any three vertices $u,v,w$ it holds that if there is an edge directed from $u$ to $v$ and an edge directed from $v$ to $w$, then there is an edge between $u$ and $w$ and it is directed from $u$ to $w$.
As every comparability graph is perfect, in particular its chromatic number and clique number coincide~\cite{graph-classes}.

\begin{lemma}\label{lem:left-of-is-poset}
 For $M = \M{+}{0}{0}{-}$, every $p \geq 0$, and every ordered graph $G$, the graph $M_p(G)$ is a comparability graph and hence $\chi(M_p(G)) = \omega(M_p(G))$.
\end{lemma}
\begin{proof}
 Let $p \geq 0$ and $G$ be fixed.
 Two edges $e_1,e_2$ of $G$ are conflicting if the right endpoint of one edge, say $e_1$, lies at least $p$ positions left of the left endpoint of the other edge $e_2$.
 In this case we orient the edge $\{e_1,e_2\}$ in the conflict graph $M_p(G)$ from $e_1$ to $e_2$.
 Clearly, if $e_1,e_2$ are conflicting with $e_1$ being at least $p$ positions left of $e_2$ and $e_2,e_3$ are conflicting with $e_2$ being at least $p$ positions left of $e_3$, then also $e_1,e_3$ are conflicting with $e_1$ being at least $p$ positions left of $e_3$.
 Hence we have defined a transitive orientation of $M_p(G)$, proving that $M_p(G)$ is a comparability graph.
\end{proof}

For an edge $e = (u,v)$ in an ordered graph $G$ we say that the \emph{span of $e$} is the closed interval $[u,v] \subseteq \bZ$.

\begin{figure}[htb]
 \centering
 \includegraphics{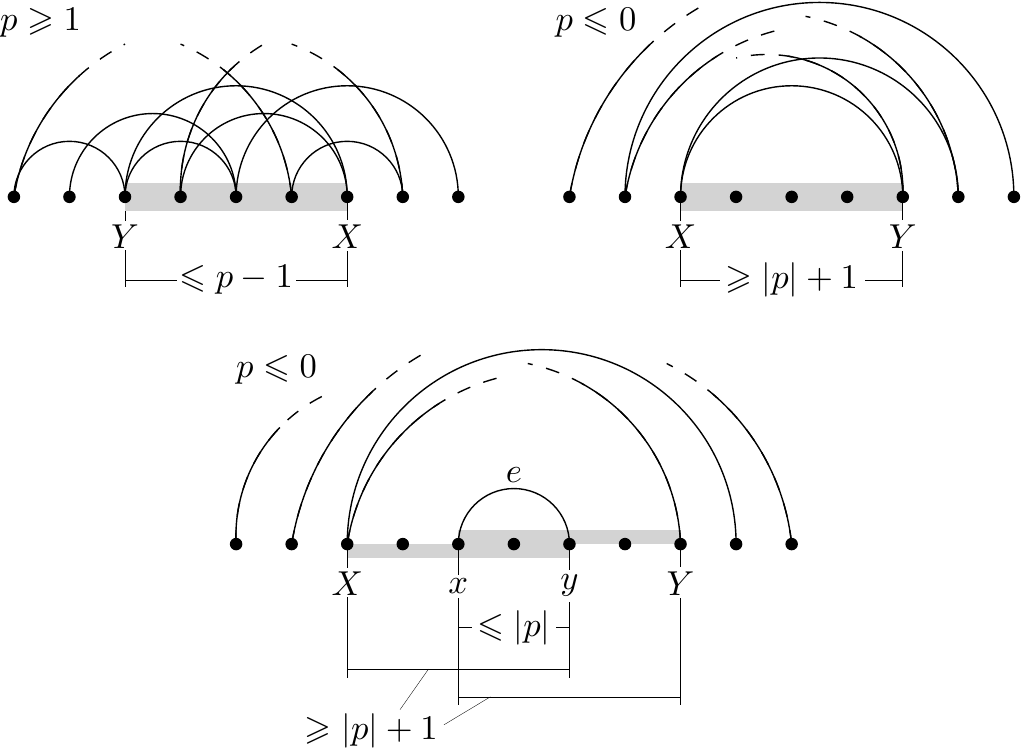}
 \caption{Illustration of independent sets of edges in $M_p(G)$ for $M = (+\>0\>0\>-)$ according to Lemma~\ref{lem:independent-sets}: condition~\ref{enum:p-geq-1} (top-left), condition~\ref{enum:p-leq-0-first} (top-right) and condition~\ref{enum:p-leq-0-second} (bottom).}
 \label{fig:left-of-independent-sets}
\end{figure}

\begin{lemma}\label{lem:independent-sets}
 Let $G$ be an ordered graph and $F \subseteq E(G)$ be a subset of edges.
 \begin{itemize}
  \item If $p \geq 1$, then $F$ is an independent set in $M_p(G)$ if and only if the following holds:
   \begin{enumerate}[label = (\roman*)]
    \item There exists a closed interval $[Y,X]$ of length at most $p-1$ intersecting the span of every edge in $F$, see the top-left of Figure~\ref{fig:left-of-independent-sets}.\label{enum:p-geq-1}
   \end{enumerate}
  \item If $p \leq 0$, then $F$ is an independent set in $M_p(G)$ if and only if one of the following holds:
   \begin{enumerate}[label = (\roman*), start = 2]
    \item There exists a closed interval $[X,Y]$ of length at least $|p|+1$ that is contained in the span of every edge in $F$, see the top-right of Figure~\ref{fig:left-of-independent-sets}.\label{enum:p-leq-0-first}
    \item There exists an edge $e = (x,y)$ in $F$ of length at most $|p|$ and a closed interval $[X,Y]$ with $X \leq y  + p - 1$ and $Y \geq x - p + 1$ that is contained in the span of every edge in $F - e$, see the bottom of Figure~\ref{fig:left-of-independent-sets}.\label{enum:p-leq-0-second}
   \end{enumerate}
 \end{itemize} 
\end{lemma}
\begin{proof}
 First note that condition~\ref{enum:p-geq-1} and~\ref{enum:p-leq-0-first} can be simultaneously rephrased as follows:
 \begin{enumerate}[label = (\roman*')]
  \item There are integers $X$ and $Y$, with $X-Y \leq p-1$, such that every edge in $F$ has left endpoint at most $X$ and right endpoint at least $Y$.\label{enum:first-kind}
 \end{enumerate}

 Assume that $F$ satisfies condition~\ref{enum:first-kind}.
 For any edge $e_1$ with left endpoint $u$, $u\leq X$, and any edge $e_2$ with right endpoint $v$, $v\geq Y$, we have $u\leq X \leq Y+p-1 \leq v+p-1$, and hence $e_1$ and $e_2$ are not conflicting.
 Now assume that $p\leq 0$ and $F$ satisfies condition~\ref{enum:p-leq-0-second}.
 As the interval $[X,Y]$ has length at least $|p|+2$ it follows from the previous argument that $F-e$ is an independent set.
 Moreover, for any edge $e' \in F-e$, $e' = (u,v)$, we have $u \leq X \leq y +p-1$ and $v \geq Y \geq x - p+1$, i.e., $e$ and $e'$ are not conflicting.
 It follows that $F$ is an independent set in $M_p(G)$.

 \medskip
 
 Now consider any independent set $F$ of $M_p(G)$.
 Let $x$ denote the rightmost left endpoint and $y$ the leftmost right endpoint of edges in $F$.
 First assume that $x-y\leq p-1$.
 Then $X=y+p-1 \geq x$ and $Y=y$ are integers with $X-Y=p-1$ such that every edge in $F$ has left endpoint at most $X$ and right endpoint at least $Y$, and hence $F$ satisfies condition~\ref{enum:first-kind}.
 Secondly, assume that $x-y\geq p$.
 Let $e_1\in F$ with left endpoint $x$ and $e_2\in F$ with right endpoint $y$.
 If $e_1\neq e_2$, then $e_1$ and $e_2$ are not in conflict and hence $x-y\leq p-1$, a contradiction.
 Therefore $e_1$ and $e_2$ are the same edge $e$.
 Hence for each edge $(u,v) \in F - e$ we have $x-v\leq p-1$ and $u-y \leq p-1$.
 With $X = y+p-1$ and $Y = x-p+1$ we have that every edge in $F - e$ has left endpoint at most $X$ and right endpoint at least $Y$, see the bottom of Figure~\ref{fig:left-of-independent-sets}.
 Finally observe that $e$ has length $y-x\leq -p$ and thus this case can happen only if $p\leq -1$.
 In particular, $F$ and $p$ satisfy condition~\ref{enum:p-leq-0-second}.
\end{proof}

The following concept was introduced by Dujmovi\'{c} and Wood~\cite{DW04} for $p=1$.
For integers $p$, $t$ with $p \geq 0$ and $t \geq 1$, an unordered graph $F = (V,E)$ is called \emph{$p$-almost $t$-colorable} if there exists a set $S \subseteq V$ of at most $p(t-1)$ vertices, such that $\chi(F - S) \leq t$.
The following result was proven by Dujmovi\'{c} and Wood~\cite{DW04} in the special case of $p=1$.
Here we prove it in general.
Recall that $G^\star$ is the underlying unordered graph of a given ordered graph $G$

\begin{lemma}\label{lem:a-almost-t-colorable}
 Let $M = \M{+}{0}{0}{-}$.
 For any $p \geq 0$ and any graph $F$ we have that
 \[
  \min \{\omega(M_p(G)) \mid G^\star = F\} = \min \{t \mid F \text{ is $p$-almost $(t+1)$-colorable}\}.
 \]
\end{lemma}
\begin{proof}
 First assume that $F$ is $p$-almost $(t+1)$-colorable.
 We will find an ordered graph $G$ with $G^\star=F$, i.e., an embedding of $V(F)$ into $\bZ$, such that $\omega(M_p(G))\leq t$.
 There is a set $S$ of at most $pt$ vertices in $F$ such that $\chi(F-S)\leq t+1$.
 Let $C_1,\ldots,C_{t+1}$ denote the color classes of a proper coloring of $F-S$ and let $S=S_1\dot\cup\cdots\dot\cup S_t$ denote a partition of $S$ with disjoint sets $S_i$ of size at most $p$ each.
%
 Set $a_0 = 0$, $a_i = |C_i \cup S_i|$ for $1\leq i\leq t$, and $a_{t+1} = |C_{t+1}|$.
 For each $i$, $1\leq i\leq t+1$, consider the interval $I_i = [a_{i-1}+1,a_{i-1}+a_i] \subset\bZ$.
 We form an ordered graph $G$ with $G^\star = F$ by bijectively mapping $C_i$ into the first $|C_i|$ vertices in $I_i$ and bijectively mapping $S_i$ into the remaining vertices in $I_i$, $1\leq i\leq t$, and mapping the vertices in $C_{t+1}$ bijectively into $I_{t+1}$.
 Observe that for two conflicting edges the right endpoint of one edge is left of the left endpoint of the other edge.
 Moreover an edge that has both endpoints in $I_i$ has its right endpoint in $S_i$, as $C_i$ is an independent set.
 Hence two edges having left endpoints in $I_i$ are not in conflict, since either $p=0$ and $S_i=\emptyset$, or the distance between the copies of any two vertices from $S_i$ in $G$ is at most $p-1$, $1\leq i\leq t$.
 Moreover no edge has left endpoint in $I_{t+1}$.
 Therefore a maximum clique in $M_p(G)$ has at most $t$ vertices.
 It follows that $\omega(M_p(G)) \leq t$, as desired.
 
 \medskip

 Now assume that $G$ is an ordered graph with $G^\star = F$ and $\omega(M_p(G)) = t$.
 We shall show that $F$ is $p$-almost $(t+1)$-colorable.
 By Lemma~\ref{lem:left-of-is-poset} the vertices of $M_p(G)$ can be split into $t = \omega(M_p(G))$ independent sets $E_1,\ldots,E_t$.
 If $p \geq 1$, by Lemma~\ref{lem:independent-sets}~\ref{enum:p-geq-1} there is a closed interval $I_i\subset\bZ$ of length at most $p-1$ that intersects the span of each edge in $E_i$, $i=1,\ldots,t$.
 If $p=0$, by Lemma~\ref{lem:independent-sets}~\ref{enum:p-leq-0-first} there is a closed interval $I'_i\subset\bZ$ of length at least $1$ that is contained in the span of each edge in $E_i$, $i=1,\ldots,t$.
 Therefore we can choose a closed interval $I_i\subset(I'_i\setminus\bZ)$ (of length $<1$) that intersects the span of each edge in $E_i$, $i=1,\ldots,t$.
 
 \begin{figure}[htb]
  \centering
  \includegraphics{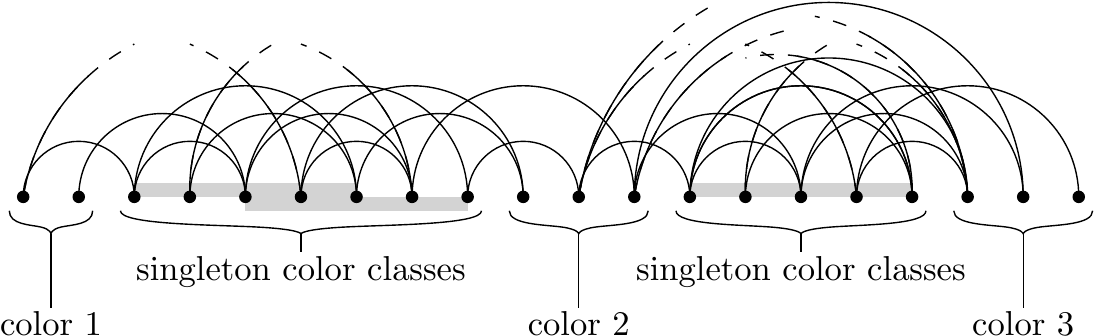}
  \caption{A $p$-almost $(t+1)$-coloring for $p = 5$ and $t = 3$.}
  \label{fig:p-almost-coloring}
 \end{figure}

 We define a coloring $c$ of $G$ as follows.
 See Figure~\ref{fig:p-almost-coloring} for an illustration.
 Every vertex that is contained in some $I_i$, $i=1,\ldots,t$, defines a singleton color class.
 Note that there are at most $pt$ such vertices (if $p=0$ the intervals $I_i$ contain no vertices).
 The remaining vertices of $G$ are split by the intervals $I_i$ into at most $t+1$ consecutive sets of integers and we color all vertices in such a set in the same color, using at most $t+1$ further colors.
 
 The coloring $c$ is a proper coloring of $G$, since a monochromatic edge $e$ would have a span that is disjoint from all intervals $I_1,\ldots,I_t$ and hence $e$ would not be contained in any of $E_1,\ldots,E_t$, a contradiction.
 Hence $G$ (and thus $F$) is $p$-almost $(t+1)$-colorable, as desired.
\end{proof}

Having Lemma~\ref{lem:left-of-is-poset} and~\ref{lem:a-almost-t-colorable}, we are now ready to determine $A(M,p,k)$ and $W(M,p,k)$ for $M = \M{+}{0}{0}{-}$.

\begin{proposition}[edges at distance at least $p$, row $10$ in Table~\ref{tab:detailed}]\label{P00M}
 ~\\
 Let $M \in \{\M{+}{0}{0}{-}, \M{0}{-}{+}{0}\}$.
 \begin{compactitem}[\enskip]
  \item If $p \geq 1$, then $A(M,p,k) = \floor{\frac{(k+1)^2}{4}}-1$ and $W(M,p,k) = \ceil{\frac{k-1}{p+1}}$ for all $k \geq 2$.
 
  \item If $p \leq 0$, then $A(M,p,k) = 1$ for $2\leq k \leq |p|+1$ and $A(M,p,k) = \floor{\frac{(k-|p|)^2}{4}}$ for $k > |p|+1$ and $W(M,p,k) = k-1$ for all $k \geq 2$.
 \end{compactitem}
\end{proposition}
\begin{proof}
 Consider $M=\M{+}{0}{0}{-}$. 
 
 Two edges $e_1 = (u_1,v_1)$, $e_2 = (u_2,v_2)$ are in conflict if and only if $u_1-p\geq v_2$ or $u_2-p\geq v_1$.
 If $p \geq 1$, then the spans of conflicting edges are disjoint and at least $p$ positions apart, see Figure~\ref{fig:Example-Matrices} top right.
 In particular, a clique in $M_p(G)$ is a set of edges with pairwise disjoint spans at distance at least $p$.
 In case $p=0$, spans of conflicting edges are only interiorly disjoint, i.e., they intersect in at most one point.
 If $p < 0$, then two edges are conflicting if their spans become disjoint after one edge is shifted $|p|+1$ positions to the right. 
 Note that if $e_1$ and $e_2$ both have length at most $|p|-1$, then this might hold no matter which edge is shifted.
 Here a clique in $M_p(G)$ is a set of edges in which any pair of edges has disjoint spans after shifting one edge $|p|+1$ positions right.

 \medskip
 
 Consider $A(M,p,k)$ for $p \geq 1$.
 Consider $G = K_k$ with vertex set $V = \{ip \mid i\in[k]\}$.
 Then for any two edges $e_1,e_2 \in E$ with $e_1 = (u_1,v_1)$ and $e_2 = (u_2,v_2)$ we have
 \[
   u_1 > v_2 \quad \Rightarrow \quad u_1 - p \geq v_2.
 \]
 Hence, by Lemma~\ref{lem:independent-sets}~\ref{enum:p-geq-1}, every independent set in $M_p(G)$ is an edge set in $G$ with pairwise intersecting spans.
 For $i=1,\ldots,k$, the number of edges in $G$ whose span contains the vertex $ip \in V$ is given by
 \begin{align*}
  \binom{k}{2} - \binom{i-1}{2} - \binom{k-i}{2} &= \frac{k(k-1) - (i-1)(i-2) - (k-i)(k-i-1)}{2}\\
  &= \frac{-2i^2 + 2i - 2 + 2ik}{2} = i(k+1-i) -1.
 \end{align*}
 Note that for $n \in \{0,1\}$ we have $\binom{n}{2} = n(n-1)/2 = 0$.
 As this is maximized for $i = \ceil{(k+1)/2}$, we conclude that $A(M,p,k) \leq \alpha(M_p(G)) = \ceil{(k+1)/2}\floor{(k+1)/2}-1 = \floor{(k+1)^2/4}-1$.

 \medskip
 
 Now consider any ordered graph $G$ with $\chi(G) = k$.
 Let $G'$ denote a $k$-critical subgraph of $G$.
 Then $G'$ has minimum degree at least $k-1$.
 Then $M_p(G')$ is an induced subgraph of $M_p(G)$ and hence $\alpha(M_p(G))\geq \alpha(M_p(G'))$.
 Consider the set $S$ of the $\ceil{(k+1)/2}$ leftmost vertices in $G'$ and let $v$ be the rightmost vertex in $S$.
 Every vertex in $S \setminus v$ has at least $k-1 - (\ceil{(k+1)/2}-1) = \floor{(k+1)/2}-1$ edges to the right of $S$.
 Moreover, vertex $v$ has at least $k-1$ incident edges.
 In total this is a set $I$ of at least
 \[
  \left(\ceil{\frac{k+1}{2}}-1\right)\left(\floor{\frac{k+1}{2}}-1\right) + k-1 = \floor{\frac{(k+1)^2}{4}}-1
 \]
 edges in $G'$.
 Choosing $X = Y = v$ (as $p \geq 1$ we have $X - Y \leq p-1$) shows that $I$ is an independent set by Lemma~\ref{lem:independent-sets}~\ref{enum:p-geq-1}.
 Thus $\alpha(M_p(G)) \geq \alpha(M_p(G')) \geq |I| \geq \floor{(k+1)^2/4}-1$.
 As $G$ was arbitrary, we get $A(M,p,k) \geq \floor{(k+1)^2/4}-1$.
 

 \medskip

 Next consider $W(M,p,k)$ for $p \geq 0$.
 Consider $G = K_k$ on vertex set $[k]$.
 Recall that a clique in $M_p(G)$ is a set of edges in $G$ that are pairwise at least $p$ positions apart of each other.
 Thus a largest clique $C$ in $M_p(G)$ is formed by taking every $(p+1)^{\text{th}}$ edge of length $1$ in $G$.
 As there are $k-1$ edges of length $1$ in total, it follows that $W(M,p,k) \leq \omega(M_p(G)) = \ceil{(k-1)/(p+1)}$, as desired.
 
 \medskip
 
 Now fix $G$ to be any ordered graph with $\chi(G) = k$.
 Then by Lemma~\ref{lem:a-almost-t-colorable} we have that $G$ is $p$-almost $(\omega(M_p(G))+1)$-colorable.
 In particular, $k = \chi(G) \leq (p+1)\omega(M_p(G)) + 1$, which gives $\omega(M_p(G)) \geq \ceil{(k-1)/(p+1)}$.
 As $G$ was arbitrary we get $W(M,p,k) \geq \ceil{(k-1)/(p+1)}$.

 \medskip 

 Now consider $A(M,p,k)$ for $p \leq 0$.
 Recall that each independent set is of one of two kinds due to Lemma~\ref{lem:independent-sets}~\ref{enum:p-leq-0-first} and~\ref{enum:p-leq-0-second}.
 Consider $G = K_k$ on vertex set $[k]$.
 If $k \leq |p|+1$, then every edge has length at most $|p|$.
 As every pair of non-conflicting edges has an edge of length at least $|p|+1$ (c.f.\ Figure~\ref{fig:left-of-independent-sets}), we have in this case that $A(M,p,k) = \alpha(M_p(G)) = 1$.
 If $k > |p|+1$, then consider for each $X = 1,\ldots,k+p$ all edges with left endpoint at most $X$ and right endpoint at least $Y=X-p+1$.
 As $p \leq 0$ we have $X < Y$, see the top-right of Figure~\ref{fig:left-of-independent-sets}.
 There are exactly $X(k-(Y-1)) = X(k+p-X)$ such edges and this term is maximized for $X = \ceil{(k+p)/2}$.
 Hence any independent set of the first kind contains at most $\floor{(k+p)^2/4}$ elements.
 Finally, it is easy to see that, since $k>|p|+1$, for any independent set of the second kind one can replace the short edge by some edge of length $|p|+1$ to obtain an independent set of the first kind that has the same number of edges.
 Together we have that $A(M,p,k) \leq \alpha(M_p(G)) = \floor{(k+p)^2/4}$, as desired.

 \medskip
 
 Now consider any ordered graph $G$ with $\chi(G) = k$.
 Let $G'$ denote a $k$-critical subgraph of $G$.
 Then $G'$ has minimum degree at least $k-1$.
 Then $M_p(G')$ is an induced subgraph of $M_p(G)$ and hence $\alpha(M_p(G))\geq \alpha(M_p(G'))$.
 Consider $i = \ceil{(k+p)/2}$ and $j = \ceil{(k-p)/2}+1$, and let $X$ and $Y$ be the integers corresponding to the $i$-th and $j$-th vertex in $G'$ counted from the left, respectively.
 Then $X-Y\leq \ceil{(k+p)/2}-\ceil{(k-p)/2}-1=p-1$.
 The set $I$ of all edges in $G'$ with left endpoint at most $X$ and right endpoint at least $Y$ is an independent set of the first kind by Lemma~\ref{lem:independent-sets}~\ref{enum:p-leq-0-first}.
 As $p \leq 0$, we have $i < j$ and hence $I$ consists of at least $i(k-1-(j-2)) = i(k-j+1)$ edges, since $\delta(G')\geq k-1$.
 Thus $\alpha(M_p(G)) \geq \alpha(M_p(G')) \geq |I| \geq i(k-j+1) = \floor{(k+p)^2/4}$, and as $G$ was arbitrary we get $A(M,p,k) \geq \floor{(k+p)^2/4}$.

 \medskip

 Finally consider $W(M,p,k)$ for $p \leq -1$.
 Consider $G = K_k$ on vertex set $V = \{i|p| \mid i=1,\ldots,k\}$.
 Then, for any two edges $e_1,e_2 \in E(G)$ with $e_1 = (u_1,v_1)$ and $e_2 = (u_2,v_2)$ we have
 \[
  u_1 - p \geq v_2 \quad \Leftrightarrow \quad u_1 \geq v_2.
 \]
 In particular, for $G' = K_k$ with vertex set $[k]$ we have that $M_p(G)$ is isomorphic to $M_0(G')$ and thus due to the arguments above we get $W(M,p,k) \leq \omega(M_p(G))  = \omega(M_0(G')) = k-1$.
 
 \medskip
 
 Now fix $G = (V,E)$ to be any ordered graph with $\chi(G) = k$.
 As $p < 0$ we clearly have for any two edges $e_1,e_2 \in E$ with $e_1 = (u_1,v_1)$ and $e_2 = (u_2,v_2)$ that
 \[
   u_1 \geq v_2 \quad \Rightarrow \quad u_1 - p \geq v_2.
 \]
 In particular, $\omega(M_p(G)) \geq \omega(M_0(G))$ and we get $\omega(M_p(G)) \geq \omega(M_0(G)) \geq k-1$ as before.
 As $G$ was arbitrary this gives $W(M,p,k) \geq k-1$.
 
 \medskip 

 Finally, if $M' = \M{0}{-}{+}{0}$, then $M'$ is obtained from $M=\M{+}{0}{0}{-}$ by switching the first with the third and the second with the last column.
 Thus $W(M',p,k)=W(M,p,k)$ and $A(M',p,k)=A(M,p,k)$ due to Observation~\ref{obs:swapping-columns}.
\end{proof}

\begin{proposition}[right end of one edge at least $p$ positions to the right of the left end of the other edge, row $11$ in Table~\ref{tab:detailed}]\label{M00P}
 ~\\
 Let $M \in \{\M{-}{0}{0}{+}, \M{0}{+}{-}{0}\}$.
 \begin{compactitem}[\enskip]
  \item If $p \leq 1$, then $A(M,p,k) = 1$ and $W(M,p,k) = \binom{k}{2}$ for all $k \geq 2$.
 
  \item If $p \geq 2$, then $A(M,p,k) = 1$ for all $k \geq 2$ and $W(M,p,k) = 1$ if $2\leq k\leq p$ and $W(M,p,k) =\binom{k-p+2}{2}$ if $k\geq p+1$.
 \end{compactitem}
\end{proposition}
\begin{proof}

 Consider $M=\M{-}{0}{0}{+}$.
 We have $A(M,p,k) = 1$ for all $p,k \in \bZ$, $k \geq 2$ and, if $p\leq 1$, $W(M,p,k) = \binom{k}{2}$ for all $k\geq 2$ due to Theorem~\ref{thm:generalMatrices}~\ref{enum:m2m40} and~\ref{enum:m2m4greater0}.
 This leaves to consider $W(M,p,k)$ in the case $p\geq 2$.

 Here edges $(u_1,v_1)$ and $(u_2,v_2)$ are conflicting if $v_2-u_1\geq p$ or $v_1-u_2 \geq p$, that is, the right endpoint of one edge is at least $p$ steps to the right of the left endpoint of the other edge.
 A clique in $M_p(G)$ is a set of edges in $G$, where pairwise the right endpoint of one edge is at least $p$ steps to the right of the left endpoint of the other edge.
 

 \medskip

 Consider $G=K_k$ with vertex set $[k]$.
 If $k\leq p$, then for any pair of edges $(u_1,v_1)$, $(u_2,v_2)$ we have $v_2-u_1\leq k-1\leq p-1$.
 Hence no edges are conflicting and $W(M,p,k) \leq \omega(M_p(G)) = 1$.
 Therefore $W(M,p,k)=1$.
 
 If $k\geq p+1$, then the set of all edges of length at least $p-1$ in $G$ forms a clique in $M_p(G)$.
 We claim that this set is a largest clique in $M_p(G)$.
 Indeed, consider a clique $F$ in $M_p(G)$ containing an edge $e=(u_1,v_1)$ of length at most $p-2$.
 Let $e_1=(u_1-1,v_1)$, $e_2=(u_1,v_1+1)$, if they exist in $G$.
 Note that at least one of these edges exists, since $k\geq p+1$.
 Then $e_1,e_2\not\in F$, since they are not in conflict with $e$.
 Observe that if $f=(u_2,v_2) \in F \setminus e$, then $v_1-u_2\geq p$ or $v_2-u_1\geq p$.
 In the first case $e_1$ and $f$ are conflicting since $v_1-u_2\geq p$, and $e_2$ and $f$ are conflicting since $v_1+1-u_2\geq p$.
 In the second case $e_1$ and $f$ are conflicting since $v_2-(u_1-1)\geq p$, and $e_2$ and $f$ are conflicting since $v_2-u_1\geq p$.
 Thus, we can replace $e$ in $F$ with a longer edge, $e_1$ or $e_2$, and obtain a clique with at least as many edges as $F$.
 Repeating this as long as needed eventually yields a clique of size at least $|F|$ with all edges of length at least $p-1$. 
 Hence we see that the set of all edges of length at least $p-1$ in $G$ is at least as large as any other clique in $M_p(G)$.
 Recall that there are $k-\ell$ edges of length $\ell$ in this $G$, $\ell = 1,\ldots,k-1$.
 Thus $W(M,p,k) \leq \omega(M_p(G)) \leq \sum_{\ell=p-1}^{k-1}(k-\ell)=\binom{k-p+2}{2}$.
 
 \medskip
 
 Now consider any ordered graph $G$ with $\chi(G)= k \geq p+1$.
 As mentioned above the set of all edges of length at least $p-1$ in $G$ forms a clique in $M_p(G)$.
 Hence $\omega(M_p(G)) \geq\binom{k-p+2}{2}$ due to Lemma~\ref{lem:LongEdges}.
 This shows that $W(M,p,k) \leq \binom{k-p+2}{2}$.

 \medskip
 
 Finally, if $M' = \M{0}{+}{-}{0}$, then $M'$ is obtained from $M=\M{-}{0}{0}{+}$ by switching the first with the third and the second with the last column.
 Thus $W(M',p,k)=W(M,p,k)$ and $A(M',p,k)=A(M,p,k)$ due to Observation~\ref{obs:swapping-columns}.
\end{proof}

%
%
%
%
%
%
%
%

\section{Proof of Theorem~\ref{thm:nest}}\label{sec:nest}
 

We shall prove later in Lemma~\ref{lem:shift-complement} that for any ordered graph $G$ and any $p \in \bZ$ we have that $M_p(G)$ for $M=\MM{+}{0}{-}{0}{0}{+}{0}{-}$ is the complement of $M^\text{nest}_{1-p}(G)$ for $M^\text{nest}=\MM{+}{0}{-}{0}{0}{-}{0}{+}$.
Let us refer to Figure~\ref{fig:Example-Matrices} top middle for an illustration of $M^\text{nest}$.
It will hence follow that $A(M,p,k) = W(M^\text{nest},1-p,k)$ and $W(M,p,k) = A(M^\text{nest},1-p,k)$ for any $p,k \in \bZ$, $k \geq 2$.
Thus we can restrict ourselves in this section to the matrix $M=\MM{+}{0}{-}{0}{0}{+}{0}{-}$ instead of $M^\text{nest}=\MM{+}{0}{-}{0}{0}{-}{0}{+}$.
The case of $A(M,p,k)$ for $p=0$ of the following result has been considered by Dujmovi\'{c} and Wood~\cite{DW04}.

\begin{proposition}[shift by at least $p$, row $12$ in Table~\ref{tab:detailed}]\label{P0M00P0M}
 ~\\
 Let $M = \MM{+}{0}{-}{0}{0}{+}{0}{-}$.
 \begin{compactitem}[\enskip]
  \item If $p \geq 1$, then  $A(M,p,k) = k-1$ and $W(M,p,k) = \ceil{\frac{k-1}{p}}$ for all $k \geq 2$.
 
  \item If $p \leq 0$, then $\frac{k}{4(|p|+1)} \leq A(M,p,k) \leq  \ceil{\frac{k-1}{2(|p|+1)}}$ and $W(M,p,k) = 2k-3$ for all $k \geq 2$.
 \end{compactitem}
\end{proposition}
\begin{proof}
 If $p>0$, then two edges $(u_1,v_1)$ and $(u_2,v_2)$ are conflicting if either $u_1-u_2\geq p$ and $v_1-v_2\geq p$, or $u_2-u_1\geq p$ and $v_2-v_1\geq p$, that is, one edge is obtained from the other by moving each vertex at least $p$ steps to the right.
 If $p\leq 0$, then two edges $(u_1,v_1)$ and $(u_2,v_2)$ are conflicting if either $u_2-u_1\leq |p|$ and $v_2-v_1\leq |p|$, or $u_1-u_2\leq |p|$ and $v_1-v_2\leq |p|$, that is, one edge is obtained form the other by moving each vertex at most $|p|$ steps to the left or arbitrarily many steps to the right (and keeping the ordering of the vertices within the edge).
 That is, the edges $(u_1,v_1)$ and $(u_2,v_2)$ are not conflicting if and only if $|u_1-u_2|\geq |p|+1$, $|v_1-v_2|\geq |p|+1$ and the edges are nested (i.e., $u_1 < u_2 < v_2 < v_1$ or $u_2 < u_1 < v_1 < v_2$).

 \medskip

 Consider $A(M,p,k)$ for $p \geq 1$.
 Consider $G = K_k$ with vertex set $\{ip \mid i\in[k]\}$.
 Observe that any two edges of the same length are conflicting.
 Since there are only $k-1$ different lengths of edges in $G$ we have $A(M,p,k) \leq \alpha(M_p(G)) \leq k-1$.

 Now consider an arbitrary ordered graph $G$ with $\chi(G)=k$.
 The first row of $M$ is $M'=\M{+}{0}{-}{0}$.
 Hence, if $e_1e_2\in E(M_p(G))$, then $e_1e_2\in E(M'_p(G))$.
 So $E(M_p(G))$ is a subgraph of $E(M'_p(G))$, which implies $\alpha(M_p(G)) \geq \alpha(M'_p(G))$ and $\omega(M_p(G)) \leq \omega(M'_p(G))$.
 Thus with Proposition~\ref{P0M0} we can conclude that $A(M,p,k) \geq A(M',p,k) \geq k-1$, which shows that $A(M,p,k)=k-1$.

 \medskip

 Consider $W(M,p,k)$ for $p \geq 1$.
 From above we have $\omega(M_p(G)) \leq \omega(M'_p(G))$ for $M'=\M{+}{0}{-}{0}$ and any ordered graph $G$, which implies with Proposition~\ref{P0M0} that $W(M,p,k) \leq W(M',p,k) = \ceil{(k-1)/p}$.
 
 For the lower bound $W(M,p,k) \geq \ceil{(k-1)/p}$ we consider the matrix $M''=\M{+}{0}{0}{-}$.
 For any ordered graph $G$ and two edges $e_1 = (u_1,v_1)$, $e_2 = (u_2,v_2)$ in $G$ with $e_1e_2 \in E(M''_{p-1}(G))$, say $u_1 - v_2 \geq p-1$, we have
 \[
   u_1 - u_2 \geq u_1 - v_2 + 1 \geq p-1 + 1 = p
 \]
 and
 \[
   v_1 - v_2 \geq u_1 + 1 - v_2 \geq p-1 + 1 = p.
 \]
 Hence $e_1e_2 \in E(M_p(G))$ and thus $E(M''_{p-1}(G))$ is a subgraph of $E(M_p(G))$.
 As before, we conclude with Proposition~\ref{P00M} that $W(M,p,k) \geq W(M'',p-1,k) = \ceil{(k-1)/p}$.

 \medskip

 Consider $A(M,p,k)$ for $p \leq 0$ and let $q=|p|$.
 Consider $G=K_k$ with vertex set $[k]$.
 Suppose that $F$ is an independent set of size $i$ in $M_p(G)$, i.e., the edges in $F$ are pairwise nested by at least $q+1$ positions.
 Then the distance between the leftmost left endpoint and the rightmost left endpoint of edges in $F$ is at least $(i-1)(q+1)$.
 Similarly the distance between the rightmost right endpoint and the leftmost right endpoint is at least $(i-1)(q+1)$.
 Thus $2((i-1)(q+1)+1)\leq k$.
 Therefore any independent set has size at most $\floor{\frac{k-2}{2(q+1)}} + 1 = \ceil{\frac{k-1}{2(q+1)}}$.
 Thus we have $A(M,p,k) \leq \alpha(M_p(G)) \leq \ceil{\frac{k-1}{2(q+1)}}$.

 Now consider any ordered graph $G$ with $n$ vertices and $\alpha(M_p(G)) \leq a$.
 Let $G'$ be the ordered graph with vertex set $[n]$ obtained from mapping the $i^\text{th}$ vertex of $G$ to $i\in[n]$.
 Then the vertices in $G$ are in the same order as their images in $G'$ and the distance between two vertices in $G'$ is at most the distance of the corresponding preimages in $G$.
 Hence, if two edges in $G'$ are not in conflict, then the two corresponding edges in $G$ are not in conflict.
 Therefore $\alpha(M_p(G'))\leq \alpha(M_p(G)) \leq a$.

 We will show that $G'$ has fewer than $2a(q+1)n$ edges.
 For every edge $(u,v)$ of $G'$ consider its midpoint $(u+v)/2$.
 The set of possible midpoints is given by $X=\{\frac{i}{2}\mid i=3,\ldots,2n-1\}$.
 If some $\ell_x$ edges of $G'$ have the same midpoint $x\in X$, then taking every $(q+1)^{\text{st}}$ such edge (in increasing order of their lengths) gives an independent set in $M_p(G')$.
 It follows that $\ell_x \leq \alpha(M_p(G'))(q+1)$ for every midpoint $x\in X$.
 Since $|X| = 2n-3$ and $\alpha(M_p(G')) \leq a$ this gives 
 \begin{align}|E(G)| = |E(G')| \leq (2n-3)a(q+1) < 2a(q+1)n.\label{eq:degenerate}\end{align}
 
 If $H$ is an induced ordered subgraph of $G$, then $\alpha(M_p(H))\leq \alpha(M_p(G))\leq a$.
 Hence $H$ has less than $2a(q+1)|V(H)|$ edges (as the arguments above hold for any ordered graph).
 In particular $H$ has a vertex of degree less than $4a(q+1)$.
 This shows that $G$ is $(4a(q+1)-1)$-degenerate and hence $\chi(G)\leq 4a(q+1)$.
 As $G$ was arbitrary we conclude that $\Xa(M,p,a) \leq 4a(q+1)$ and using~\eqref{eq:X-from-A} we get $A(M,p,k) \geq \frac{k}{4(q+1)}$.
  
 \medskip
 
 Consider $W(M,p,k)$ for $p \leq 0$.
 Dujmovi\'{c} and Wood~\cite{DW04} prove that $W(M,p,k)=2k-3$ for $p = 0$.
 
 Consider $p\leq -1$ and any fixed ordered graph $G$.
 Let $G'$ denote the ordered graph obtained from $G$ by multiplying every vertex by $|p|+1$, i.e., the order of vertices in $G$ and $G'$ is the same, but in $G'$ vertices have pairwise distance at least $|p|+1$.
 Then two edges in $G'$ form an edge in $M_p(G')$ if and only if the corresponding edges in $G$ form an edge in $M_0(G)$.
 In particular, $M_p(G') = M_0(G)$.
 Choosing $G$ to be an ordered graph with $\chi(G) = k$ and $\omega(M_0(G))=2k-3$ (for example $G = K_k$ works) shows that $W(M,p,k) \leq \omega(M_p(G')) = \omega(M_0(G)) = 2k-3$.
 
 On the other hand, consider any ordered graph $G$ and any two edges $e_1 = (u_1,v_1)$ and $e_2 = (u_2,v_2)$ that are conflicting in $M_0(G)$, say $u_2-u_1 \leq 0$ and $v_2-v_1 \leq 0$.
 Then we have $u_2 - u_1 \leq |p|$ and $v_2-v_1 \leq |p|$, i.e., $e_1$ and $e_2$ are also conflicting in $M_p(G)$.
 This shows that $\omega(M_p(G)) \geq \omega(M_0(G))$ for any $G$ and hence $W(M,p,k) \geq W(M,0,k) = 2k-3$.
\end{proof}

\begin{lemma}\label{lem:shift-complement}
 Let $p\in\bZ$, let $G$ be an ordered graph and let $M=\MM{+}{0}{-}{0}{0}{+}{0}{-}$.
 Then the graph $M^\text{nest}_p(G)$ is the complement of the graph $M_{1-p}(G)$.
\end{lemma}
\begin{proof}
 Let $e_1 = (u_1,v_1)$, $e_2=(u_2,v_2)$ be two edges in $G$.
 If $e_1e_2 \in E(M_{1-p}(G))$, say $M(u_1,v_1,u_2,v_2)^\top \geq \mathbf{1-p}$, then
 \begin{align*}
  & \hspace{-3em} u_1-u_2 \geq 1-p \land v_1-v_2 \geq 1-p\\
  \Rightarrow & \hspace{0.5em} \big(\lnot\ u_1-u_2 \leq -p\big) \land \big(\lnot\ v_1-v_2 \leq -p\big)\\
  \Rightarrow & \hspace{0.5em} \big(\lnot\ u_2-u_1 \geq p\big) \land \big(\lnot\ v_2-v_1 \geq p\big)\\
  \Rightarrow & \hspace{0.5em} e_1e_2 \notin E(M^\text{nest}_{p}(G)).
 \end{align*}
 Similarly, if $e_1e_2 \in E(M^\text{nest}_{p}(G))$, say $M^\text{nest}(u_1,v_1,u_2,v_2)^\top \geq \mathbf{p}$, then
 \begin{align*}
  & \hspace{-3em} u_1-u_2 \geq p \land v_2-v_1 \geq p\\
  \Rightarrow & \hspace{0.5em} u_2-u_1 \leq -p \land v_1-v_2 \leq -p\\
  \Rightarrow & \hspace{0.5em} \big(\lnot\ u_2-u_1 \geq 1-p\big) \land \big(\lnot\ v_1-v_2 \geq 1-p\big)\\
  \Rightarrow & \hspace{0.5em} e_1e_2 \notin E(M_{1-p}(G)).
 \end{align*}
\end{proof}
 
Lemma~\ref{lem:shift-complement} shows that for any $p,k \in \bZ$, $k \geq 2$, we have $A(M^\text{nest},p,k) = W(M,1-p,k)$ and $W(M^\text{nest},p,k) = A(M,1-p,k)$.
Therefore Theorem~\ref{thm:nest} follows from Proposition~\ref{P0M00P0M}.\qed

%

\section{Conclusions}\label{sec:conclusions}

In this paper we consider ordered graphs and introduce the notion of conflicting pairs of edges with respect to a fixed  matrix $M \in \bZ^{s \times 4}$ and a parameter $p \in \bZ$.
This algebraic framework captures many interesting graph parameters, such as the page-number, queue-number or interval chromatic number.
We consider the following extremal question for given $M$ and $p$:

\begin{center}
 \raggedright
 \textit{``What is the maximum chromatic number $\Xw(M,p,w)$, respectively $\Xa(M,p,a)$, among all ordered graphs with no set of $w$ pairwise conflicting edges, respectively no set of $a$ pairwise non-conflicting edges?''}
\end{center}

We give sufficient conditions on the pairs of matrices $M$ and $p\in\bZ$ under which $\Xw(M,p,w)$ and/or $\Xa(M,p,a)$ are as small or as large as possible for any $a$, $w\geq 1$; namely when $\Xw(M,p,w)=f(w)$ or $\Xw(M,p,w) = \infty$, respectively $\Xa(M,p,a)=f(a)$ or $\Xa(M,p,a) = \infty$ (recall that for $x \in \bZ$, $x \geq 1$, $f(x)$ is the largest integer $k$ with $\binom{k}{2}\leq x$).
Moreover, we give exact results for all $1 \times 4$--matrices with entries in $\{-1,0,1\}$.
Note that additionally to the results from Theorem~\ref{thm:nest} exact values for several $\medmuskip=0mu 2\times4$--matrices can be obtained from Theorems~\ref{thm:generalMatrices} and~\ref{thm:specialMatrices} using Lemma~\ref{lem:small-facts}~\ref{enum:exchange}.

\medskip

\begin{figure}[htb]
 \centering
 \includegraphics{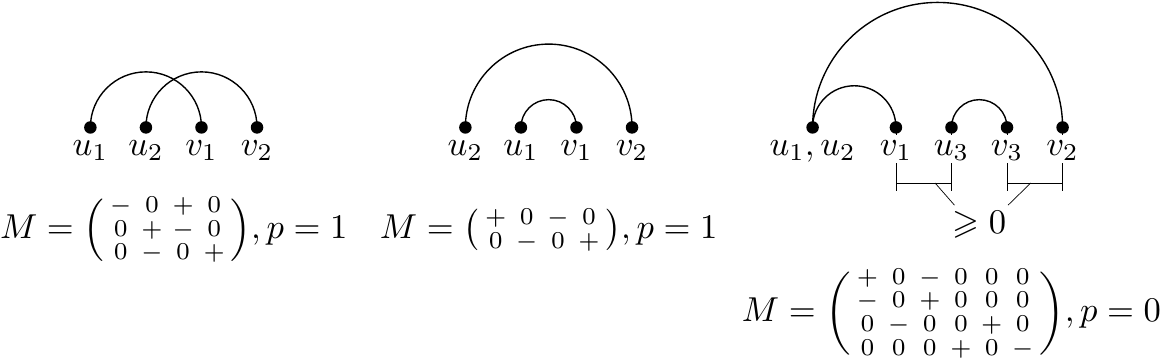}
 \caption{Matrix $M$ and parameter $p$ corresponding to the presence of a pair of crossing edges (left), nesting edges (middle) and a bonnet (right).}
 \label{fig:Example-Matrices-2}
\end{figure}

Determining $\Xw(M,p,w)$ and $\Xa(M,p,a)$ exactly for more matrices $M$ and parameters $p$ remains an interesting challenge, for example for the ``cross''-matrix in the left of Figure~\ref{fig:Example-Matrices-2} or the ``nesting''-matrix in the center of Figure~\ref{fig:Example-Matrices-2}.
In addition, several more general questions remain open.
Most noticeable, all our lower bounds are attained by complete graphs and hence it would be interesting to find $M$ and $p$ for which the maximum chromatic number $\Xw(M,p,w)$ or $\Xa(M,p,a)$ is \emph{not} attained by any complete graph.
More specifically we have the following question.

\begin{question}\label{que:CompleteNotTight}
 Are there integers $s$, $p$, $t$, and a matrix $M\in\bZ^{s\times 4}$ such that for every complete ordered graph $G$ on $k$ vertices we have $\alpha(M_p(G))> A(M,p,k)$ or $\omega(M_p(G))> W(M,p,k)$?
 What if $s=1$ or $s = 2$?
\end{question}

Our framework can be naturally extended to conflicts that are defined on sets of $t \geq 3$ edges, rather than just pairs of edges, in which case one would use matrices $M \in \bZ^{s \times 2t}$.
Then the maximum chromatic number among all ordered graphs not containing a particular ordered graph on $t$ edges as an ordered subgraph is given by $\Xw(M,0,1)$ for an appropriate matrix $M \in \bZ^{s \times 2t}$.
Most recently, the authors have shown the existence of ordered graphs of arbitrarily large chromatic number without so-called bonnets~\cite{our-trees}.
In terms of the framework here, this can be restated as $\Xw(M,0,1) = \infty$, where $M$ is the $4 \times 6$--matrix in the right of Figure~\ref{fig:Example-Matrices-2} (where, in contrast to~\cite{our-trees}, a triangle is considered as a bonnet).
It is easy to see that for any ordered graph $G$ that contains a triangle we have $\omega(M_0(G))\geq 2$.
Therefore the graphs that yield $\Xw(M,0,1) = \infty$ are not complete graphs (compare with Question~\ref{que:CompleteNotTight}).

\medskip

Another natural generalization of the framework is obtained by considering other parameters of the conflict graph.
For example one may ask for the maximum chromatic number among all ordered graphs with conflict graphs of bounded density, bounded chromatic number or small maximum degree.

\bibliographystyle{plain}
\bibliography{lit-2}

\end{document}